\documentclass[12pt]{article}
\usepackage{amsmath,amscd,amsbsy,amssymb,latexsym,url,bm,amsthm}
\usepackage[vlined,boxed,commentsnumbered,linesnumbered,ruled]{algorithm2e}
\usepackage{epsfig,graphicx,subfigure}
\usepackage{enumitem,balance,mathtools}
\usepackage{wrapfig}
\usepackage{mathrsfs, euscript}
\usepackage[usenames]{xcolor}
\usepackage{hyperref}
\usepackage{epstopdf}
\usepackage{float}
\usepackage{comment}
\author{Fei Wang\thanks{School of Mathematical Sciences, CMA-Shanghai, Shanghai Jiao Tong University, Shanghai, China \href{mailto:fwang256@sjtu.edu.cn}{\texttt{fwang256@sjtu.edu.cn}}} \and Lingda Xu\thanks{Department of Applied Mathematics, The Hong Kong Polytechnic University, Hong Kong, China \href{mailto:lingda.xu@polyu.edu.hk}{\texttt{lingda.xu@polyu.edu.hk}}} \and Zeren Zhang\thanks{School of Mathematical Sciences, Shanghai Jiao Tong University, Shanghai, China \href{mailto:zhangzr0018@sjtu.edu.cn}{\texttt{zhangzr0018@sjtu.edu.cn}}}}

\title{The stability threshold for 3D MHD equations around Couette with rationally aligned magnetic field}
\date{}
\allowdisplaybreaks[4]
\newcommand{\sm}{non-homogeneous modes }
\newcommand{\nm}{homogeneous modes }
\newcommand{\es}{\neq NH}
\newcommand{\en}{\neq H}
\newcommand{\et}{t_1}

\newtheorem{theorem}{Theorem}[section]
\newtheorem{lemma}{Lemma}[section]
\newtheorem{proposition}{Proposition}[section]

\newtheorem{remark}{Remark}[section]

\numberwithin{equation}{section}

\begin{document}
	\maketitle
	\begin{abstract}
		We address a stability threshold problem of the Couette flow $(y,0,0)$ in a uniform magnetic fleld $\alpha(\sigma,0,1)$  with $\sigma\in\mathbb{Q}$ for the 3D MHD equations on $\mathbb{T}\times\mathbb{R}\times\mathbb{T}$. Previously, the authors in \cite{L20,RZZ25} obtained the threshold $\gamma=1$ for $\sigma\in\mathbb{R}\backslash\mathbb{Q}$ satisfying a generic Diophantine condition, where they also proved $\gamma = 4/3$ for a general $\sigma\in\mathbb{R}$. In the present paper, we obtain the threshold $\gamma=1$ in $H^N(N>13/2)$, hence improving the above results when $\sigma$ is a rational number. The nonlinear inviscid damping for velocity $u^2_{\neq}$ is also established. Moreover, our result shows that the nonzero modes of magnetic field has an amplification of order $\nu^{-1/3}$ even on low regularity, which is very different from the case considered in \cite{L20,RZZ25}.
	\end{abstract}
	\section{Introduction}
	In this paper, we consider the 3D incompressible MHD equations on $\mathbb{T}\times\mathbb{R}\times\mathbb{T}$:
	\begin{equation}
		\left\{\begin{array}{l}
			\partial_t \tilde{u}+\tilde{u} \cdot \nabla \tilde{u}-\tilde{b} \cdot \nabla \tilde{b}=-\nabla \tilde{p}+\nu \Delta \tilde{u}, \\
			\partial_t \tilde{b}+\tilde{u} \cdot \nabla \tilde{b}-\tilde{b} \cdot \nabla \tilde{u}=\mu \Delta \tilde{b}, \\
			\nabla \cdot \tilde{u}=\nabla \cdot \tilde{b}=0 ,
		\end{array}\right.
	\end{equation}
	where $\tilde{u}$ is the velocity, $\tilde{b}$ is the magnetic field, $\tilde{p}$ is the pressure, $\nu$ is the fluid viscosity, and $\mu$ is the magnetic resistivity.
	We introduce the perturbations of the Couette flow and the homogeneous magnetic flied:
	\begin{equation*}
		u=\tilde{u}-(y,0,0)^{T},b=\tilde{b}-\alpha(\sigma,0,1)^{T},
	\end{equation*}
	with $\alpha\in\mathbb{R}$ and $\sigma\in\mathbb{Q}$. Without loss of generality, we always assume $\alpha>0$ in the present paper. Then $(u,b)$ satisfies the system
	\begin{equation}\label{pmhd}
		\left\{\begin{array}{l}
			\partial_tu+y\partial_x u -\alpha\partial_{\sigma}b+u\cdot\nabla u-b\cdot\nabla b+\left(\begin{array}{l}
				u^2\\0\\0
			\end{array}\right)=2\nabla\Delta^{-1}\partial_x u^2-\nabla p^{NL}+\nu\Delta u,\\
			\partial_tb+y\partial_x b -\alpha\partial_{\sigma}u+u\cdot\nabla b-b\cdot\nabla u-\left(\begin{array}{l}
				b^2\\0\\0
			\end{array}\right)=\mu\Delta b,\\
			p^{NL}=(-\Delta)^{-1}(\partial_ju^i\partial_iu^j-\partial_jb^i\partial_ib^j),\\
			\nabla\cdot u=\nabla\cdot b=0,\\
			u(0)=u_{in},\quad b(0)=b_{in},
		\end{array}\right.
	\end{equation}
	where summation over repeated indices is used and we defined $\partial_{\sigma}=\sigma\partial_x+\partial_z$.
	To understand the underlying transition mechanisms of the above system, we are concerned with the stability threshold problem in the Sobolev spaces and formulate it as \cite{BGM17}:

	\textbf{Stability threshold}: Given $N\geq0$, determine $\gamma_1,\gamma_2\in\mathbb{R}$ such that
	\begin{align*}
		&\|(u_{in},b_{in})\|_{H^N}\ll\nu^{\gamma_1}\mu^{\gamma_2}\ \Rightarrow\ \mathrm{stability},\\
		&\|(u_{in},b_{in})\|_{H^N}\gg\nu^{\gamma_1}\mu^{\gamma_2}\ \Rightarrow\ \mathrm{possible \ instability}.
	\end{align*}
	
	Quantifying stability thresholds has emerged as an active research frontier originating from fluid dynamics studies of the Navier-Stokes (NS) equations. In the classical NS system without the magnetic fields, the stability threshold takes the form $\nu^\gamma$ for some $\gamma\geq0$. 
	A seminal contribution by Bedrossian, Germain, and Masmoudi \cite{BGM17} established the threshold $\gamma=\frac{3}{2}$ for 3D perturbations of Couette flow on the domain 	$\mathbb{T}\times\mathbb{R}\times\mathbb{T}$ in Sobolev spaces. Specifically, they demonstrated
	that initial perturbations $u_{in}$
	satisfying $\|u_{in}\|_{H^N}\ll\nu^{3/2}$ for
	$N>\frac{9}{2}$	yield global solutions remaining order of  $O(\nu^{1/2})$.
	Subsequent refinement by Wei and Zhang \cite{WZ21} improved this threshold to $\gamma=1$ in $H^2$. Due to the absence of lift-up effects, the 2D NS system exhibits fundamentally different behavior, resulting in significantly lower stability thresholds. For perturbations on $\mathbb{T}\times\mathbb{R}$, Bedrossian, Vicol, and Wang \cite{BVW18} established a threshold $\gamma=\frac{1}{2}$, with later analyses in \cite{MZ19,WZ23} sharpening this to $\gamma=\frac{1}{3}$. When considering perturbations in Gevrey classes, distinct thresholds emerge: $\gamma=1$ \cite{BGM20,BGM22} for $\mathbb{T}\times\mathbb{R}\times\mathbb{T}$ versus $\gamma=0$ \cite{BMV16,LMZ22} for $\mathbb{T}\times\mathbb{R}$.
	Stability phenomena become more complex in the bounded domains. Recent studies in \cite{CWZ20,CLWZ20,BHIW23,WZ23b,BHIW24a,BHIW24b} have elucidated how boundary layers generate  instability mechanisms. For broader perspectives on general shear flows, we refer to \cite{LZ23,LWZ20,WZZ20,CMEW20,DL20,D23,C23,ADM21,OK80,AB24}.
	
	Due to the presence of magnetic fields, the problem becomes more complicated for magnetohydrodynamic (MHD) systems. 
	While magnetic fields generally suppress the mixing mechanisms inherent in Couette flow \cite{K24,KZ23}, they paradoxically exhibit stabilizing effects at high intensities \cite{GBM09,L20,ZZZ21}. Significant progresses have been made in quantifying these phenomena: Liss \cite{L20} proved the stability threshold $\gamma=1$ for $\sigma\in\mathbb{R}\backslash\mathbb{Q}$ satisfying a generic Diophantine condition in Sobolev spaces when $\nu=\mu>0$ on $\mathbb{T}\times\mathbb{R}\times\mathbb{T}$. Using the same method, Liss claimed the threshold $\gamma=\frac{4}{3}$ for any $\sigma\in\mathbb{R}$ under the same assumption. Recently, Rao, Zhang, and Zi \cite{RZZ25} extended these results to the case $\mu\neq\nu$. For 2D case in $\mathbb{T}\times\mathbb{R}$, Chen and Zi \cite{ZZ23b} showed $\gamma=\frac{5}{6}+$ for shear flow close to Couette flow in Sobolev spaces when $\nu=\mu$. Dolce \cite{D24} proved $(\gamma_1,\gamma_2)=(\frac{2}{3},0)$ in a more general regime $0<\mu^3\lesssim\nu\leq\mu$ where the fluid effects dominate. In case $\mu\leq\nu$, Knobel \cite{K24} established the stability threshold $(\gamma_1,\gamma_2)=(\frac{1}{12},\frac{1}{2})$ if $\nu^3\lesssim\mu\leq\nu$ and instability with inflation of size $\nu\mu^{-\frac{1}{3}}$ if $\mu\lesssim\nu^3$. The first and third authors of the present paper \cite{WZ24} and Jin, Ren and Wei \cite{JRW24} generalize these results to arbitrarily $\nu,\mu\in(0,1]$ in the vorticity system and the velocity system, respectively. The nonlinear threshold $(\gamma_1,\gamma_2)=(0,1)$ in the Gevrey-2- space when $\nu=0,\mu>0$ was obtained by Zhao and Zi \cite{ZZ24}. Recently, Dolce, Knobel, and Zillinger \cite{DKZ24} proved the algebraic instability and large norm inflation of the magnetic current in Gevrey classes when $\nu>0,\mu=0$.
	
	The purpose of this paper is to improve the stability threshold to $\gamma=1$ for $\sigma\in\mathbb{Q}$ in $\mathbb{T}\times\mathbb{R}\times\mathbb{T}$ when $\nu=\mu$, which answers a question raised in \cite{L20} affirmatively. Moreover, our result shows that the nonzero modes of magnetic field have an amplification of order $\nu^{-1/3}$ even on low regularity, which is very different from the case of $\sigma\in\mathbb{R}\backslash\mathbb{Q}$ satisfying a generic Diophantine condition.
	Defining
	\begin{equation*}
		f_0=\int_{\mathbb{T}}f(x,y,z)dx,\quad\quad f_{\neq}=f-f_0,
	\end{equation*}
	Our stability result is stated as follows:
	\begin{theorem}\label{mainthe}
		Assume $\nu=\mu\in(0,1]$, $\sigma=\frac{q}{p}\in\mathbb{Q}$, $|\alpha|>8p>0$, and $N>9/2$. Let $(u_{in},b_{in})$ be the initial datum of \eqref{pmhd}. There exist constant $0<\delta_0<1$ and $\epsilon_0=\epsilon_0(N,\sigma)>0$ such that if it holds that
		\begin{equation*}
			\|(u_{in},b_{in})\|_{H^{N+2}}=\epsilon\leq\epsilon_0\nu,
		\end{equation*}
		then the profiles $U(t,X,Y,Z)=u(t,X+Yt,Y,Z)$ and $B(t,X,Y,Z)=b(t,X+Yt,Y,Z)$ satisfy the following stability estimates:
		\begin{subequations}\label{bdmain}
			\begin{align}
				&\|e^{\delta_0 \nu^{1/3}t}(\partial_X,\partial_Z)\partial_XU^1_{\neq }\|_{L^\infty H^{N-2}}+\nu^{1/6}\|(\partial_X,\partial_Z)\partial_XU^1_{\neq }\|_{L^2H^{N-2}}\lesssim\epsilon,\label{bdu1}\\
				&{\|e^{\delta_0\nu^{1/3}t}(\partial_X,\partial_Z)\nabla_L U^2_{\neq }\|_{L^\infty H^{N-2}}+\|U^2_{\neq }\|_{L^2 H^{N-2}}+\nu^{1/6}\|(\partial_X,\partial_Z)\nabla_L  U^2_{\neq }\|_{L^2 H^{N-2}}\lesssim\epsilon},\label{bdu2}\\
				&\|e^{\delta_0 \nu^{1/3}t}(\partial_X,\partial_Z)^2U^3_{\neq }\|_{L^\infty H^{N-2}}+\nu^{1/6}\|(\partial_X,\partial_Z)^2U^3_{\neq }\|_{L^2H^{N-2}}\lesssim\epsilon,\label{bdu3}\\
				&\|e^{\delta_0\nu^{1/3}t}\partial_XB^1_{\neq }\|_{L^\infty H^{N}}+\nu^{1/6}\|(\partial_X,\partial_Z)\partial_{X}B^1_{\neq }\|_{L^2 H^{N}}\lesssim\nu^{-1/3}\epsilon,\label{bdb1}\\
				&\|e^{\delta_0\nu^{1/3}t}(\partial_X,\partial_Z)(B^2_{\neq },B^3_{\neq })\|_{L^\infty H^{N}}+\nu^{1/6}\|(\partial_X,\partial_Z)(B^2_{\neq },B^3_{\neq})\|_{L^2 H^{N}}\lesssim\epsilon,\label{bdb2}\\
				&\|(1,\partial_Z)^2(U_0,B_0)\|_{L^{\infty}H^N}\lesssim\epsilon,\label{bdu0}
			\end{align}
		\end{subequations}		
		where $\nabla_L=(\partial_X,\partial_Y-t\partial_X,\partial_Z)$.
	\end{theorem}
	\begin{remark}
		The inviscid damping of $U^2_{\neq}$ is captured by the uniform $\nu$-bounds in \eqref{bdu2}. Since $B^2_{\neq}$ does not exhibit the inviscid damping, there is an amplification of order $\nu^{-1/3}$ in the right side of \eqref{bdb1}. Note that $(\partial_X,\partial_Z)\partial_XB^1_{\neq}$ exhibits a growth of order $\nu^{-1/2}$ due to the nonlinear effect. The enhanced dissipation of the nonzero modes is characterized by the $e^{\delta_0\nu^{1/3}t}$ and the $\nu^{-1/6}$ growth of the $L^2$ in time estimates in \eqref{bdu1}--\eqref{bdb2}. Furthermore, the estimate \eqref{bdu0} describes the suppression of the lift-up effect for zero-mode.
	\end{remark}
	\begin{remark}
		Recalling the equations \eqref{pmhd}, the $\alpha\partial_{\sigma}(b,u)$ term exhibits the effect of background magnetic field. Using $k,l$ to represent the Fourier variable of $x,z$ respectively, for $\sigma=q/p$, there are only two cases: $|\sigma k+l|\geq1/p$ and $|\sigma k+l|=0$. Therefore, we decompose \begin{equation*}
			F_{\neq}=F_{\es}+F_{\en},
		\end{equation*} where $F_{\es}=\mathbb{P}_{\sigma k+l\neq0}F_{\neq}$ is referred as the non-homogeneous modes and $F_{\en}=\mathbb{P}_{\sigma k+l=0}F_{\neq}$ denotes the homogeneous modes (see \eqref{defcom1} and \eqref{defcom2} for details). For the \sm $(U_{\es},B_{\es})$, we have the stability estimates without regularity losses:
		\begin{align*}
			\|e^{\delta_0 \nu^{1/3}t}(\partial_X,\partial_Z)\nabla_L(U^2_{\es},B^2_{\es})\|_{L^\infty H^N}+\|&(\partial_X,\partial_Z)(U^2_{\es},B^2_{\es})\|_{L^2H^N}\\
			&+\nu^{\frac{1}{6}}\|(\partial_X,\partial_Z)\nabla_L(U^2_{\es},B^2_{\es})\|_{L^2H^N}\lesssim\epsilon,\\
			\|e^{\delta_0 \nu^{1/3}t}(\partial_X,\partial_Z)^2(U^3_{\es},B^3_{\es})\|_{L^\infty H^N}+\nu^{\frac{1}{6}}\|&(\partial_X,\partial_Z)^2(U^3_{\es},B^3_{\es})\|_{L^2H^N}\lesssim\epsilon.
		\end{align*}
		For the \nm $(U_{\en},B_{\en})$, it holds that
		\begin{align*}
			&{\|e^{\delta_0\nu^{1/3}t}\partial_X\nabla_L U^2_{\en}\|_{L^\infty H^{N-2}}+\|U^2_{\en}\|_{L^2 H^{N-2}}+\nu^{1/6}\|\partial_X\nabla_L U^2_{\en}\|_{L^2 H^{N-2}}\lesssim\epsilon},\\
			&{\|e^{\delta_0\nu^{1/3}t}\partial_{XX} U^3_{\en}\|_{L^\infty H^{N-2}} +\nu^{1/6}\|\partial_{XX}U^3_{\en}\|_{L^2 H^{N-2}}\lesssim\epsilon,}\\
			&\|e^{\delta_0\nu^{1/3}t}\partial_Y^LB^2_{\en}\|_{L^\infty H^{N}}+\nu^{1/6}\|\partial_Y^LB^2_{\en}\|_{L^2 H^{N}}\lesssim\nu^{-1/3}\epsilon,\\
			&\|e^{\delta_0\nu^{1/3}t}\partial_X(B^2_{\en},B^3_{\en})\|_{L^\infty H^{N}}+\nu^{1/6}\|\partial_X(B^2_{\en},B^3_{\en})\|_{L^2 H^{N}}\lesssim\epsilon,
		\end{align*}
		which is the reason why we have $H^{N-2}$ in \eqref{bdu1}--\eqref{bdu3} and the  $\nu^{-1/3}$ growth in \eqref{bdb1}. Using the incompressible condition, we can obtain the estimates of $(U^1_{\neq},B^1_{\neq})$.
	\end{remark}
	
	\begin{remark}
		Combining the methods of \cite{RZZ25} and the present paper, one can prove that, for $\nu\neq\mu$, if $|\alpha|\geq8p(\mu+\nu)/\sqrt{\mu\nu}$, the estimate similar to \eqref{bdmain} holds for the initial datum satisfying
		\begin{equation*}
			\|(u_{in},b_{in})\|_{H^{N+2}}=\epsilon\leq\epsilon_0\min\{\nu,\mu\}.
		\end{equation*}
	\end{remark}
	\begin{remark}
		It is still open whether the threshold can be improved for $\sigma\in \mathbb{R}\backslash\mathbb{Q}$ which does not satisfy the generic Diophantine condition.
	\end{remark}
	\noindent\textbf{Main ideas of the proof}: In this paper, we deal with \sm and \nm (see \eqref{defcom1} and \eqref{defcom2} for definitions) separately. We first examine the linear terms of the \sm $F^{\pm,i}_{\es}$ for $i\in\{2,3\}$ (see \eqref{defwf} for definition). The growth of $F^{\pm,i}_{\es}$ is caused by the linear stretching terms $\partial^L_{XY}W^{\pm,i}_{\es}$ . Motivated by \cite{WZ21}, we employ the dissipation of $\partial_XW^{\pm,3}_{\es}$ to get the $\nu^{-2/3}$ growth of $F^{\pm,3}_{\es}$. While for $F^{\pm,2}_{\es}$, we use the enhanced dissipation of the quantity $\partial_X|\nabla_L|W^{\pm,2}_{\es}$ to obtain the amplification of order $\nu^{-1/3}$.
	
	For the \nm $(U_{\en},B_{\en})$, the behaviors of $U^2_{\en}$ and $B^2_{\en}$ are essentially different. In fact, letting $Q=\Delta_LU$ and $G=\Delta_LB$, the linearization system of $(Q^2_{\en},G^2_{\en})$ reads 
	\begin{equation*}
		\left\{
		\begin{aligned}
			&\partial_tQ^2_{\en}-\nu\Delta_L Q^2_{\en}=0,\\
			&\partial_tG^2_{\en}+2\partial_{XY}^L\Delta_L^{-1}G^2_{\en}-\nu\Delta_L G^2_{\en}=0.
		\end{aligned}\right.
	\end{equation*}
	A direct calculation gives that
	\begin{equation*}
		\|Q^2_{\en}(t)\|_{L^2}\lesssim e^{-c\nu t^3}\|Q^2_{in}\|_{L^2},\quad \|G^2_{\en}(t)\|_{L^2}\lesssim \nu^{-2/3}e^{-c\nu t^3}\|G^2_{in}\|_{L^2}.
	\end{equation*} When considering the nonlinear interactions, the $\Delta_L(B\cdot \nabla_LB^2)_{\en}$ term leads to an additional amplification of order $\nu^{-1/3}$ for $Q^2_{\en}$, which causes trouble in the nonlinear estimates. For instance, we do not have a uniform bound of $U^2_{\en}$ in $H^N$. To overcome these difficulties, we exploit the subtle relationship between growth and regularity of $U^i_{\en}$ and $\Delta_LU^i_{\en}$ with $i\in\{2,3\}$ (see \eqref{bsnu} for details).
	
	Turning to the zero-mode, the most challenging part is of the term
	\begin{equation*}
		\left\langle (U^{Hi,2}_{\en}\cdot\partial_Y^LG^{Lo,1}_{\en})_0,G^1_0\right\rangle_{H^N},
	\end{equation*}
	which is caused by the absence of uniform bound of $U^2_{\en}$. To reduce the derivative losses as much as possible, we use the multiplier $M_3$ introduced in \cite{WZ23}  (see section \ref{secf0} for details). 
	\section{Preliminaries}
	\subsection{Notations}
	Throughout this paper, for $a,b\in\mathbb{R}$, we define
	\begin{equation*}
		|a,b|:=\sqrt{a^2+b^2},\quad\langle a\rangle:=\sqrt{1+|a|^2}.
	\end{equation*}
	We use the notation $a\lesssim b$ to express $a\leq Cb$ for some constant $C>0$ independent of the parameters of interest such as $\nu,\alpha,\sigma$. The Fourier transform of a function $f$ is denoted by
	\begin{equation*}
		\mathcal{F}(f)(k,\eta,l)=\hat{f}(k,\eta,l)=\iint_{\mathbb{T}\times\mathbb{R}\times\mathbb{T}}e^{-i(kx+\eta y+lz)}f(x,y,z)\mathrm{d}x\mathrm{d}y\mathrm{d}z.
	\end{equation*}
	In this paper, for $\sigma=q/p$, there are only two cases: $|\sigma k+l|\geq1/p$ and $|\sigma k+l|=0$. Therefore, we decompose \begin{equation*}
		f_{\neq}=f_{\es}+f_{\en},
	\end{equation*} where the \sm $f_{\es}$ represents the projection onto the frequencies in $x,z$ satisfying $|\sigma k+l|>0$, denoted by
	\begin{equation}\label{defcom1}
		f_{\es}=\mathbb{P}_{\sigma k+l\neq0}f_{\neq}:=\mathcal{F}^{-1}(\mathbf{1}_{\sigma k+l\neq0,k\neq0}\hat{f}(k,\eta,l)),
	\end{equation}
	while the \nm $f_{\en}$ is given by
	\begin{equation}\label{defcom2}
		f_{\en}=\mathbb{P}_{\sigma k+l=0}f_{\neq}:=\mathcal{F}^{-1}(\mathbf{1}_{\sigma k+l=0,k\neq0}\hat{f}(k,\eta,l)).
	\end{equation}
	At times it is also convenient to project $f$ onto the nonzero frequencies in $z$. For this, we use the alternate notation
	\begin{equation*}
			\mathbb{P}_{l\neq 0}f:=\mathcal{F}^{-1}(\mathbf{1}_{l\neq0}\hat{f}(k,\eta,l)).
	\end{equation*}
	We define the oscillation multiplier given in \cite{L20}:
	\begin{equation*}
		T^{t}_{a}f:=\mathcal{F}^{-1}(e^{ia(\sigma k+l)t}\hat{f}(k,\eta,l)).
	\end{equation*}
	We also introduce the paraproduct decomposition
	\begin{align*}
		fg=&\ \mathcal{F}^{-1}\sum_{k',l'\in\mathbb{Z}}\int_{\eta'\in\mathbb{R}}\hat{f}(k',\eta',l')\hat{g}(k-k',\eta-\eta',l-l')\chi(k,\eta,l,k',\eta',l')\mathrm{d}\eta'\\
		&+\mathcal{F}^{-1}\sum_{k',l'\in\mathbb{Z}}\int_{\eta'\in\mathbb{R}}\hat{f}(k',\eta',l')\hat{g}(k-k',\eta-\eta',l-l')(1-\chi(k,\eta,l,k',\eta',l'))\mathrm{d}\eta'\\
		=&:f^{Hi}g^{Lo}+f^{Lo}g^{Hi},
	\end{align*}
	where 
	\begin{equation*}
		\chi(\xi\in\mathbb{R}^3,\xi'\in\mathbb{R}^3)=\left\{\begin{array}{l}
			1\quad\mathrm{if}\ |\xi-\xi'|\leq2|\xi'|,\\
			0\quad\mathrm{otherwise}.
		\end{array}\right.
	\end{equation*}
	For a function of space and time $f(t,x,y,z)$ defined on a time interval $(a,b)$, we denote the space $L^p(a,b;H^s)$ for $1\leq p\leq\infty$ by the norm
	\begin{equation*}
		\|f\|_{L^p(a,b;H^s)}:=\|\|f\|_{H^s}\|_{L^p(a,b)}.
	\end{equation*}
	When the time interval is clear from context or mentioned explicitly elsewhere, we use the shorthand notation $\|f\|_{L^p(a,b;H^s)}=\|f\|_{L^p H^s}$.
	\subsection{Reformulation of the equations}
	\textbf{Change of independent variables}: First, we change the coordinates to mod out by the fast mixing of the Couette flow:
	\begin{equation*}
		X=x-yt,\quad Y=y,\quad Z=z.
	\end{equation*}
	Denoting $U(t,X,Y,Z)=u(t,x,y,z)$ and $B(t,X,Y,Z)=b(t,x,y,z)$, the system of $(U,B)$ reads
	\begin{equation*}
		\left\{
		\begin{array}{l}
			\partial_tU-\alpha\partial_\sigma B+U\cdot\nabla_LU-B\cdot\nabla_LB+\left(\begin{array}{l}
				U^2\\0\\0
			\end{array}\right)=2\nabla_L\Delta_L^{-1}\partial_XU^2-\nabla_LP^{NL}+\nu\Delta_LU,\\
			\partial_tB-\alpha\partial_\sigma U+U\cdot\nabla_LB-B\cdot\nabla_LU-\left(\begin{array}{l}
				B^2\\0\\0
			\end{array}\right)=\nu\Delta_LB,\\
			P^{NL}=(-\Delta_L)^{-1}(\partial^L_jU^i\partial^L_iU^j-\partial^L_jB^i\partial^L_iB^j),
		\end{array}\right.
	\end{equation*}
	where summation over repeated indices is used, $\nabla_L=(\partial_X,\partial^L_Y,\partial_Z)=(\partial_X,\partial_Y-t\partial_X,\partial_Z),\Delta_L=\nabla_L\cdot\nabla_L$, and $\partial_\sigma=\sigma\partial_X+\partial_Z$. In general, for a function $\phi(t,x,y,z)$, we denote $\Phi(t,X,Y,Z)=\phi(t,x,y,z)$.\\	
	\textbf{Change of dependent variables}: For the \sm $(U_{\es},B_{\es})$ and zero-mode $(U_0,B_0)$, following \cite{L20}, we introduce good unknowns to symmetrize the system. Denoting
	\begin{equation}\label{defwf}
		W^{\pm}=T^{-t}_{\pm\alpha}(U\pm B),\quad F^{\pm}=\Delta_LW^{\pm},
	\end{equation}
	 where multiplier $T^{-t}_{\pm\alpha}$ given by $T^{-t}_{\pm\alpha}(t,k,\eta,l)=e^{\mp i\alpha (\sigma k+l)t}$, the system of $W^{\pm}$ and $F^{\pm}$ read
	\begin{equation*}
		\left\{
		\begin{aligned}
			\partial_tW^++(T^{-t}_{2\alpha }W^-)\cdot\nabla_L W^++\left(\begin{array}{l}
				T^{-t}_{2\alpha }W^{-,2}\\0\\0
			\end{array}\right)=&\nabla_L\Delta_L^{-1}\partial_X(W^{+,2}+T^{-t}_{2\alpha }W^{-,2})\\
			&+\nabla_L\Delta_L^{-1}(\partial^L_j(T^{-t}_{2\alpha }W^{-,i})\partial^L_iW^{+,j})+\nu\Delta_LW^+,\\
			\partial_tW^-+(T^{-t}_{-2\alpha }W^+)\cdot\nabla_L W^-+\left(\begin{array}{l}
				T^{-t}_{-2\alpha }W^{+,2}\\0\\0
			\end{array}\right)=&\nabla_L\Delta_L^{-1}\partial_X(W^{-,2}+T^{-t}_{-2\alpha }W^{+,2})\\
			&+\nabla_L\Delta_L^{-1}(\partial^L_j(T^{-t}_{-2\alpha }W^{+,i})\partial^L_iW^{-,j})+\nu\Delta_LW^-,
		\end{aligned}	\right.
	\end{equation*}
	and
	\begin{equation}\label{eqf}
		\left\{\begin{aligned}
			\partial_tF^{\pm,1}&+2\partial^L_{XY}\Delta_L^{-1}F^{\pm,1}-\partial_{XX}\Delta_L^{-1}(F^{\pm,2}+T^{-t}_{\pm2\alpha}F^{\mp,2})+T^{-t}_{\pm2\alpha}F^{\mp,2}-\nu\Delta_LF^{\pm,1}=-(T^{-t}_{\pm2\alpha}W^{\mp})\cdot\nabla_LF^{\pm,1}\\
			&-(T^{-t}_{\pm2\alpha}F^{\mp})\cdot\nabla_LW^{\pm,1}-2\partial^L_i(T^{-t}_{\pm2\alpha}W^{\mp,j})\partial^L_{ij}W^{\pm,1}+\partial_X(\partial^L_j(T^{-t}_{\pm2\alpha}W^{\mp,i})\partial^L_iW^{\pm,j}),\\
			\partial_tF^{\pm,2}&+\partial^L_{XY}\Delta_L^{-1}F^{\pm,2}-\partial^L_{XY}\Delta_L^{-1}T^{-t}_{\pm2\alpha}F^{\mp,2}-\nu\Delta_LF^{\pm,2}=-(T^{-t}_{\pm2\alpha}W^{\mp})\cdot\nabla_LF^{\pm,2}\\
			&-(T^{-t}_{\pm2\alpha}F^{\mp})\cdot\nabla_LW^{\pm,2}-2\partial^L_i(T^{-t}_{\pm2\alpha}W^{\mp,j})\partial^L_{ij}W^{\pm,2}+\partial_Y^L(\partial^L_j(T^{-t}_{\pm2\alpha}W^{\mp,i})\partial^L_iW^{\pm,j}),\\
			\partial_tF^{\pm,3}&+2\partial^L_{XY}\Delta_L^{-1}F^{\pm,3}-\partial^L_{XZ}\Delta_L^{-1}(F^{\pm,2}+T^{-t}_{\pm2\alpha}F^{\mp,2})-\nu\Delta_LF^{\pm,3}=-(T^{-t}_{\pm2\alpha}W^{\mp})\cdot\nabla_LF^{\pm,3}\\
			&-(T^{-t}_{\pm2\alpha}F^{\mp})\cdot\nabla_LW^{\pm,3}-2\partial^L_i(T^{-t}_{\pm2\alpha}W^{\mp,j})\partial^L_{ij}W^{\pm,3}+\partial_Z(\partial^L_j(T^{-t}_{\pm2\alpha}W^{\mp,i})\partial^L_iW^{\pm,j}),\\
		\end{aligned}
		\right.
	\end{equation}
	respectively.
	To treat the linear stretching term $\partial^L_{XY}\Delta_L^{-1}F^{\pm,2}_{\es}$ of the \sm $F^{\pm,2}_{\es}$, we introduce the quantity $\partial_X|\nabla_L|W^{\pm,2}_{\es}$, which satisfies
	\begin{equation}\label{eqsymv2}
		\begin{aligned}
			\partial_t(\partial_X|\nabla_L|W^{\pm,2}_{\es})&+\frac{\partial^L_{XXY}}{|\nabla_L|}T^{-t}_{\pm2\alpha}W^{\mp,2}_{\es}-\nu\Delta_L\partial_X|\nabla_L|W^{\pm,2}_{\es}=\\
			&-\partial_X|\nabla_L|((T^{-t}_{\pm2\alpha}W^{\mp})\cdot\nabla_LW^{\pm,2})_{\es}+\frac{\partial^L_{XY}}{|\nabla_L|}(\partial^L_j(T^{-t}_{\pm2\alpha }W^{\mp,i})\partial^L_iW^{\pm,j})_{\es},
		\end{aligned}
	\end{equation}
	where $|\nabla_L|=\sqrt{-\Delta_L}$.
	For the \nm $(U_{\en},B_{\en})$, considering that the linearization of $(U_{\en},B_{\en})$ is decoupled, we treat the system of $(U_{\en},B_{\en})$ directly. 
	Let $Q=\Delta_LU,G=\Delta_LB$. Then $(U_{\en},B_{\en})$ and
	$(Q_{\en},G_{\en})$ satisfy the equations
	\begin{equation}\label{equb}
		\left\{
		\begin{array}{l}
			\partial_tU_{\en}+(U\cdot\nabla_LU)_{\en}-(B\cdot\nabla_LB)_{\en}+\left(\begin{array}{l}
				U^2_{\en}\\0\\0
			\end{array}\right)=2\nabla_L\partial_{X}\Delta_L^{-1}U^2_{\en}-\nabla_LP^{NL}_{\en}+\nu\Delta_LU_{\en},\\
			\partial_tB_{\en}+(U\cdot\nabla_LB)_{\en}-(B\cdot\nabla_LU)_{\en}-\left(\begin{array}{l}
				B^2_{\en}\\0\\0
			\end{array}\right)=\nu\Delta_LB_{\en}.\\
			P^{NL}=(-\Delta_L)^{-1}(\partial^L_jU^i\partial^L_iU^j-\partial^L_jB^i\partial^L_iB^j),
		\end{array}\right.
	\end{equation}
    and 
	\begin{equation}\label{eqqh}
		\left\{
		\begin{aligned}
			\partial_tQ_{\en}+2\partial_{XY}^L\Delta_L^{-1}Q_{\en}-&2\nabla_L\partial_X\Delta_L^{-1}Q^2_{\en}+\Delta_L(U\cdot\nabla_LU)_{\en}-\Delta_L(B\cdot\nabla_LB)_{\en}\\
			&+\left(\begin{array}{l}
				Q^2_{\en}\\0\\0
			\end{array}\right)=
			\nabla_L(\partial^L_jU^i\partial_i^LU^j-\partial^L_jB^i\partial^L_iB^j)_{\en}+\nu\Delta_L Q_{\en},\\
			\partial_tG_{\en}+2\partial_{XY}^L\Delta_L^{-1}G_{\en}+&\Delta_L(U\cdot\nabla_LB)_{\en}-\Delta_L(B\cdot\nabla_LU)_{\en}-\left(\begin{array}{l}
				G^2_{\en}\\0\\0
			\end{array}\right)=\nu\Delta_L G_{\en},
		\end{aligned}\right.
	\end{equation}
	respectively.
	\subsection{The Fourier multipliers}
	We introduce the following multipliers to capture the stability mechanisms of the above systems. The first two multipliers $M^1$ and $M^2$ are now standard for studying the stability of Couette flow. They are determined by the equations
	\begin{equation*}
		\left\{\begin{array}{l}
			\displaystyle\frac{\partial_tM_1}{M_1}=-\frac{k^2+|kl|}{k^2+(\eta-kt)^2+l^2},\ \mathrm{for}\ k\neq0,\\
			\displaystyle\frac{\partial_tM_2}{M_2}=-\frac{\nu^{1/3}k^2}{k^2+\nu^{2/3}(\eta-kt)^2},\ \mathrm{for}\ k\neq0,\\
			M_j(t,0,\eta,l)=M_j(0,k,\eta,l)=1,\ j=1,2.
		\end{array}\right.
	\end{equation*}
	The multiplier $M_1$ helps us to obtain inviscid damping and control terms arising from the linear pressure. The multiplier $M_2$ is used to capture the enhanced dissipation, which satisfies the key property:
	\begin{equation}\label{prom3}
		\nu^{1/3}\lesssim\nu(\eta-kt)^2-\frac{\partial_tM_2}{M_2}.
	\end{equation}
	We also need a multiplier $M_3$ introduced in \cite{WZ23}, which helps us to control the nonlinear forcing cascades on the zero-mode in this paper, defined as
	\begin{equation*}
		\left\{\begin{array}{l}
			\displaystyle\frac{\partial_tM_3}{M_3}=-\Upsilon(t,k,\eta,l),\\
			M_3(0,k,\eta,l)=1,
		\end{array} 
		\right.
	\end{equation*} 
	with
	\begin{equation*}
			\Upsilon(t,k,\eta,l)=\sum_{k'\neq k}\frac{1}{|k-k'|}\frac{1+|k-k'|+|k'|}{(1+|k-k'|+|k'|)^2+(\eta-(k-k')t)^2}.\label{defup}
	\end{equation*}
	Letting $\delta_0=\frac{1}{100|\alpha|}$, we define the main multiplier as
	\begin{equation*}
		M(t,k,\eta,l)=\left\{\begin{aligned}
			&e^{\delta_0\nu^{1/3}t}(M_1M_2)(t,k,\eta,l),\ \mathrm{for}\ k\neq 0,\\
			&M_3(t,0,\eta,l),\ \mathrm{for}\ k=0.
		\end{aligned}\right.
	\end{equation*}
	\subsection{Bootstrap argument}\label{secbs}
	To prove Theorem \ref{mainthe}, we use a bootstrap argument. By a standard local well-posedness argument, we state the following lemma without showing more details. 
	\begin{lemma}\label{st}
		Assume $\nu=\mu\in(0,1]$, $\sigma=\frac{q}{p}\in\mathbb{Q}$, $\alpha>8p$, and $N>9/2$. Let $(u_{in},b_{in})$ be the initial datum of \eqref{pmhd}. There exists $t_0\ll1$ independent of $\nu$ such that if $\|(u_{in},b_{in})\|_{H^{N+2}}\leq\epsilon$, then there holds
		\begin{equation*}
			\sup_{t\in[0,2t_0]}\|(U(t),B(t))\|_{H^{N+2}}\leq2\epsilon.
		\end{equation*}
	\end{lemma}
	From now on, all time norms are taken over the interval $[t_0,\et]$ unless otherwise stated.\\
		\textbf{Bootstrap hypotheses:}
		We introduce two norms 
		\begin{align*}
			&\|f\|_{A^N}:=\|f\|_{L^{\infty}H^N}+\nu^{1/2}\|\nabla_Lf\|_{L^2H^N}+\|\partial_X|\nabla_L|^{-1}f\|_{L^2H^N}+\nu^{1/6}\|f\|_{L^2H^N},\\
			&\|g\|_B:=\|g\|_{L^{\infty}H^N}+\nu^{1/2}\|\nabla_Lg\|_{L^2H^N}+\|\sqrt{-\Upsilon}g\|_{L^2H^N}.
		\end{align*}
		Fix $C_0$ a large constant determined by the proof below and let $\et> t_0$ be the maximal time such that the following estimates hold on $[t_0,\et]$:
		\begin{enumerate}
			\item The \sm bounds:
			\begin{subequations}\label{bss}
				\begin{align}
					&\|(\partial_X,\partial_Z)|\nabla_L|MW^2_{\es}\|_{A^N}\leq8\epsilon,\label{bssymv2}\\
					&\|MF^2_{\es}\|_{A^N}\leq 8C_0\nu^{-1/3}\epsilon,\label{bsf2}\\
					&\|(\partial_X,\partial_Z)^2MW^3_{\es}\|_{A^N}\leq8C_0\epsilon,\label{bsz3}\\
					& \|MF^3_{\es}\|_{A^N}\leq8C_0^2\nu^{-2/3}\epsilon\label{bsf3};
				\end{align}
			\end{subequations}
			\item The \nm of velocity bounds:
			\begin{subequations}\label{bsnu}
				\begin{align}
					&\|MQ^2_{\en}\|_{A^N}\leq8\nu^{-1/3}\epsilon,\label{bsq2}\\
					&\|\langle t\rangle^{-1}MQ^2_{\en}\|_{A^{N-1}}\leq8\epsilon,\label{bsq2l}\\
					&\nu^{1/3}\|\partial_{XX }MU^3_{\en}\|_{A^N}+\|\partial_{XX}MU^3_{\en}\|_{A^{N-2}}\leq8C_0\epsilon,\label{bsu3}\\
					&\nu^{1/3}\|MQ^3_{\en}\|_{A^N}+\|MQ^3_{\en}\|_{A^{N-2}}\leq8C_0^2\nu^{-2/3}\epsilon;\label{bsq3}
				\end{align}
			\end{subequations}
			\item The \nm of magnetic field bounds
			\begin{subequations}\label{bsnb}
				\begin{align}
					&\|\partial_XMB^2_{\en}\|_{A^N}+\nu^{1/6}\|\partial_{XX}MB^2_{\en}\|_{A^N}\leq8\epsilon,\label{bsb2}\\
					&\nu^{1/3}\|\partial^L_YMB^2_{\en}\|_{A^N}+\nu^{1/2}\|\partial^L_{XY}MB^2_{\en}\|_{A^N}+\nu^{2/3}\|MG^2_{\en}\|_{A^N}\leq8C_0\epsilon,\label{bsh2}\\ 
					&\nu^{1/3}\|\partial_{XX}MB^3_{\en}\|_{A^N}+\|\partial_XMB^3_{\en}\|_{A^N}\leq8C_0\epsilon,\label{bsb3}\\
					&\nu^{1/3}\|MG^3_{\en}\|_{A^{N}}+\|MG^3_{\en}\|_{A^{N-2}}\leq8C_0^2\nu^{-2/3}\epsilon;\label{bsh3}
				\end{align}
			\end{subequations}
			\item The zero-mode bounds:
			\begin{subequations}\label{bs0}
			\begin{align}
				&\|MF^1_0\|_B\leq8\nu^{-2/3}\epsilon,\ \|MF^2_0\|_B\leq8\nu^{-1/3}\epsilon,\label{bsf0}\ \|MF^3_0\|_B\leq8\nu^{-2/3}\epsilon,\\
				&\|(1,\partial_Z)^2MW_0\|_B\leq8\epsilon\label{bsz0}.
			\end{align}
		\end{subequations}
		\end{enumerate}
	Next, we prove that $\et =+\infty$ in the proposition below.
	\begin{proposition}\label{bsup1}
		Under the assumptions of Theorem \ref{mainthe} and bootstrap hypotheses, there exists $0<\epsilon_0=\epsilon_0(N,\sigma)$ such that if $\epsilon<\epsilon_0\nu$, then estimates \eqref{bss}-\eqref{bs0} hold with all the constants on the right-hand side divided by 2 on $[t_0,\et]$. It follows by continuity that $\et=+\infty$.
	\end{proposition}
	
	The proof of Proposition \ref{bsup1} constitutes the majority of our work. Before proving it, let us briefly comment on the structure of bootstrap hypotheses. For the \sm, the hypothesis \eqref{bss} is consistent with the result in \cite{L20} but without the regularity losses. Turning to the \nm of velocity, distinct regularity treatments are required for different quantities. As noted earlier, we anticipate the $\nu^{-1/3}$ growth of $Q^2_{\en}$ \eqref{bsq2} in $H^N$. The linear pressure term $\partial^4_{X}\Delta_L^{-1}U^2_{\en}$ leads to an amplification of order $\nu^{-1/3}$ of $\partial_{X}^2U^3_{\en}$ in \eqref{bsu3}. By sacrificing one degree of regularity, we achieve the uniform bound of $\langle t\rangle^{-1}Q^2_{\en}$ in \eqref{bsq2l}, thereby securing the uniform bound of $\partial_X^2U^3_{\en}$ in $H^{N-2}$. 
	For the \nm of magnetic fields, while the presence of pressure term restricts our estimates to the quantity $Q^2_{\en}$, its absence enables us to deal with $\nabla_LB^2_{\en}$ and $\nabla_L^2B^2_{\en}$ in \eqref{bsb2} and \eqref{bsh2}. In particular, the nonlinear stretching term $ \partial_X^2(B^1_{\en}\partial_XU^2_{\en})$ causes the $\nu^{-1/6}$ growth of $\partial_{X}^2B^2_{\en}$ in \eqref{bsb2}.
	For the zero-mode, the bound of $F^2_0$ is smaller than those for $F^1_0$ and $F^3_0$ in \eqref{bsf0}. In fact, by the divergence-free condition, $F^2_0$ always has a nonzero $Z$-frequency, which ensures that there is not nonlinear interaction between two homogeneous modes in the energy estimate of $F^2_0$. The same idea is used to obtain the bound of $\partial_{Z}^2W_0$ in \eqref{bsz0}.
	
	Finally, we prove some useful estimates that follow immediately from the bootstrap hypotheses.
	\begin{lemma}\label{est1}
		Under the bootstrap hypotheses, the following estimates hold on $[t_0,\et]$:
		\begin{subequations}
			\begin{align}
				&\|\nabla_{X,Z}\partial_XMW^1_{\es}\|_{L^\infty H^N}+\nu^{1/2}\|\nabla_L\nabla_{X,Z}\partial_XMW^1_{\es}\|_{L^2H^N}+\nu^{1/6}\|\nabla_{X,Z}\partial_XMW^1_{\es}\|_{L^2H^N}\lesssim\epsilon,\\
				&\|\partial_X^2MU^1_{\en}\|_{L^\infty H^N}+\nu^{1/2}\|\nabla_L\partial_X^2MU^1_{\en}\|_{L^2H^N}+\nu^{1/6}\|\partial_X^2MU^1_{\en}\|_{L^2H^N}\lesssim\nu^{-1/3}\epsilon,\\
				&\|\partial_X^2MU^1_{\en}\|_{L^\infty H^{N-2}}+\nu^{1/2}\|\nabla_L\partial_X^2MU^1_{\en}\|_{L^2H^{N-2}}+\nu^{1/6}\|\partial_X^2MU^1_{\en}\|_{L^2H^{N-2}}\lesssim\epsilon,\\
				&\|\partial_XMB^1_{\en}\|_{L^\infty H^N}+\nu^{1/2}\|\nabla_L\partial_XMB^1_{\en}\|_{L^2H^N}+\nu^{1/6}\|\partial_X^2MB^1_{\en}\|_{L^2H^N}\lesssim\nu^{-1/3}\epsilon,\\
				&\|\partial_X|\nabla_L|MU^2_{\en}\|_{L^2H^N}\lesssim\nu^{-1/3}\epsilon\label{estu2},\\
				&\|\partial_{XXX}U^2_{\en}\|_{L^2H^{N-2}}\lesssim\epsilon,\label{estu2l}
			\end{align}
		\end{subequations}	
		where $\nabla_{X,Z}=(\partial_X,\partial_Z)$.
	\end{lemma}
	\begin{proof}
		By divergence-free condition, for $V\in\{W,U,B\}$, we have
		\begin{align*}
			\|\partial_XV^1_{\es}\|_{H^N}\lesssim\|\partial_Y^LV^2_{\es}\|_{H^N}+\|\partial_ZV^3_{\es}\|_{H^N},\quad\|\partial_XV^1_{\en}\|_{H^N}\lesssim\|\partial_Y^LV^2_{\en}\|_{H^N}+\|\partial_XV^3_{\en}\|_{H^N}.
		\end{align*}
		Then by \eqref{prom3} and  $|k,\eta-kt,l|^{-1}\lesssim(k^2\langle t\rangle)^{-1}\langle k,\eta,l\rangle$, the estimates follow immediately from the bootstrap hypotheses.
	\end{proof}
	\section{Energy estimates on \sm}\label{secs}

	In this section, we improve the bootstrap hypothesis \eqref{bss}. To systematically analyze the nonlinear terms, we introduce the following notations. For the nonlinear transport term of $\partial_X|\nabla_L|MW^{+,2}_{\es}$ in \eqref{ee1}, we denote
	\begin{equation*}
		NLT=\int_{t_0}^{\et}\left\langle \partial_X|\nabla_L|((T^{-t}_{+2\alpha}W^{-})\cdot\nabla_LW^{+,2})_{\es},\partial_X|\nabla_L|MW^{+,2}_{\es}\right\rangle_{H^N} dt,
	\end{equation*} as well as 
	\begin{equation*}
		NLT^{D}(j,s_1,s_2)=\int_{t_0}^{\et}\left\langle \partial_X|\nabla_L|((T^{-t}_{+2\alpha}W^{-,j}_{s_1})^{D_1}(\partial^L_jW^{+,2}_{s_2})^{D_2}),\partial_X|\nabla_L|MW^{+,2}_{\es}\right\rangle_{H^N} dt,
	\end{equation*}
	with $D\in\{HL,LH\},D_1,D_2\in\{Hi,Lo\},j\in\{1,2,3\}$, and $s_1,s_2\in\{\es,\en,0\}$. 
	When indices $D,j,s_1,s_2$ are omitted or replaced by $``\cdot"$, it indicates no restrictions on those parameters. For example
	\begin{equation*}
		NLT(1)=NLT(1,\cdot,\cdot)=\int_{t_0}^{\et}\left\langle \partial_X|\nabla_L|((T^{-t}_{+2\alpha}W^{-,1})\partial_XW^{+,2})_{\es},\partial_X|\nabla_L|MW^{+,2}_{\es}\right\rangle_{H^N} dt,
	\end{equation*}
	and
	\begin{equation*}
		NLT^{LH}(\cdot,\cdot,\en)=\int_{t_0}^{\et}\left\langle \partial_X|\nabla_L|((T^{-t}_{+2\alpha}W^{-,Lo})\cdot\nabla_LW^{+,2,Hi}_{\en}),\partial_X|\nabla_L|MW^{+,2}_{\es}\right\rangle_{H^N} dt.
	\end{equation*}
	For two indices terms, such as the nonlinear pressure term of $\partial_X|\nabla_L|MW^{+,2}_{\es}$ in \eqref{ee1},
	we define the contributions as
	\begin{equation*}
	 	NLP^{D}(i,j,s_1,s_2)=\int_{t_0}^{\et}\left\langle \frac{\partial^L_{XY}}{|\nabla_L|}(\partial^L_j(T^{-t}_{+2\alpha }W^{-,i}_{ s_1})^{D_1}(\partial^L_iW^{+,j}_{ s_2})^{D_2}),\partial_X|\nabla_L|MW^{+,2}_{\es}\right\rangle_{H^N} dt\\,
	\end{equation*}
	with $D\in\{HL,LH\},D_1,D_2\in\{Hi,Lo\},i,j\in\{1,2,3\}$, and $s_1,s_2\in\{\es,\en,0\}$. Then the complete nonlinear pressure term is expressed as
	\begin{equation*}
		NLP=\sum_{i,j=1}^{3}NLP(i,j,\cdot,\cdot).
	\end{equation*} 
	Analogous notations will be adopted for other nonlinear terms.	 
	\subsection{Estimate of $(\partial_X,\partial_Z)|\nabla_L|W^{2}_{\es}$}
	In this subsection we improve \eqref{bssymv2}.
	We only consider $\partial_X|\nabla_L|W^{\pm,2}_{\es}$, and the estimate of $\partial_Z|\nabla_L|W^{\pm,2}_{\es}$ is similar. Recalling the equation of $\partial_X|\nabla_L|W^{\pm,2}_{\es}$ \eqref{eqsymv2}, it suffices to estimate $\partial_X|\nabla_L|W^{+,2}_{\es}$, with $\partial_X|\nabla_L|W^{-,2}_{\es}$ treated similarly. An energy estimate gives
	\begin{equation}\label{ee1}
		\begin{aligned}
			&\frac{1}{2}\|\partial_X|\nabla_L|MW^{+,2}_{\es}(\et)\|^2_{H^N}+\nu\|\nabla_L\partial_X|\nabla_L|MW^{+,2}_{\es}\|^2_{L^2H^N}+\sum_{i=1}^{2}\left\|\sqrt{-\frac{\partial_t M_i}{M_i}}\partial_X|\nabla_L|MW^{+,2}_{\es}\right\|^2_{L^2H^N}\\
			&=\frac{1}{2}\|\partial_X|\nabla_L|MW^{+,2}_{\es}(t_0)\|^2_{H^N}+\delta\nu^{1/3}\int_{t_0}^{\et}\|\partial_X|\nabla_L|MW^{+,2}_{\es}\|^2_{H^N}dt\\
			&\quad+\int_{t_0}^{\et}\left\langle \frac{\partial^L_{XXY}}{|\nabla_L|}T^{-t}_{2\alpha}W^{+,2}_{\es},\partial_X|\nabla_L|MW^{+,2}_{\es}\right\rangle_{H^N} dt\\
			&\quad-\int_{t_0}^{\et}\left\langle \partial_X|\nabla_L|M((T^{-t}_{2\alpha}W^{-})\cdot\nabla_LW^{+,2})_{\es},\partial_X|\nabla_L|MW^{+,2}_{\es}\right\rangle_{H^N} dt\\
			&\quad+\int_{t_0}^{\et}\left\langle \frac{\partial^L_{XY}}{|\nabla_L|}M(\partial^L_j(T^{-t}_{2\alpha }W^{-,i})\partial^L_iW^{+,j})_{\es},\partial_X|\nabla_L|MW^{+,2}_{\es}\right\rangle_{H^N} dt\\
			&=:\frac{1}{2}\|\partial_X|\nabla_L|MW^{+,2}_{\es}(t_0)\|^2_{H^N}+L_{\lambda}+OLS+NLT+NLP.
		\end{aligned}
	\end{equation}
	Noting $\delta_0$ sufficiently small and \eqref{prom3}, we obtain 
	\begin{equation}\label{estlambda}
		L_{\lambda}\leq\delta_0\left(\nu\|\nabla_L\partial_X|\nabla_L|MW^{+,2}_{\es}\|^2_{L^2H^N}+\left\|\sqrt{-\frac{\partial_t M_2}{M_2}}\partial_X|\nabla_L|MW^{+,2}_{\es}\right\|^2_{L^2H^N}\right).
	\end{equation}
	The treatment of $OLS$ is similar to that in \cite{L20}. For the sake of completeness, we give a sketch of proof. By integration by parts in time, it holds that
	\begin{align*}
		OLS&=-\frac{1}{2\alpha}\int_{t_0}^{\et}\left\langle \partial^L_{XY}\partial_\sigma^{-1}(\partial_tT^{-t}_{2\alpha})MW^{-,2}_{\es},\partial_{XX}MW^{+,2}_{\es}\right\rangle_{H^N}dt\\
		&\leq\frac{1}{2\alpha}\Bigg(2\left\|\left\langle \partial^L_{XY}\partial_\sigma^{-1}T^{-t}_{2\alpha}MW^{-,2}_{\es},\partial_{XX}MW^{+,2}_{\es}\right\rangle_{H^N}\right\|_{L^{\infty}_t}\\
		&\quad\qquad+\left|\int_{t_0}^{\et}\left\langle\partial_{XX}\partial_\sigma^{-1}T^{-t}_{2\alpha}MW^{-,2}_{\es},\partial_{XX}MW^{+,2}_{\es}\right\rangle_{H^N} dt\right|\\
		&\quad\qquad+\left|\int_{t_0}^{\et}\left\langle\partial_{XY}^L\partial_\sigma^{-1}T^{-t}_{2\alpha}(\partial_t{M}W^{-,2}_{\es}+M\partial_tW^{-,2}_{\es}),\partial_{XX}MW^{+,2}_{\es}\right\rangle_{H^N} dt\right|\\
		&\quad\qquad+\left|\int_{t_0}^{\et}\left\langle\partial_{XY}^L\partial_\sigma^{-1}T^{-t}_{2\alpha}MW^{-,2}_{\es},\partial_{XX}(\partial_t{M}W^{+,2}_{\es}+M\partial_tW^{+,2}_{\es})\right\rangle_{H^N} dt\right|\Bigg)\\
		&=:OLS_1+OLS_2+OLS_3+OLS_4.
	\end{align*}
	For $OLS_1$ and $OLS_2$, by the fact that $|\sigma k+l|\geq 1/p$, $|T^{-t}_{2\alpha}|=1$, $|\alpha|>8p$,
	and bootstrap assumption \eqref{bssymv2}, we obtain that
	\begin{align*}
		&OLS_1+OLS_2\leq \frac{p}{2|\alpha|}\left(\|\partial_X|\nabla_L|MW^{\pm,2}_{\es}\|_{L^\infty H^N}^2+\|\partial_{XX}MW^{\pm,2}_{\es}\|_{L^2 H^N}^2\right)\leq \epsilon^2/2.
	\end{align*}
	Using the equation of $W^{-,2}_{\es}$, we get
	\begin{align*}
		OLS_3\leq&\ \frac{1}{2\alpha}\Bigg(\delta_0\nu^{1/3}\|\partial_X|\nabla_L|MW^{+,2}_{\es}\|^2_{L^2H^N}+\sum_{i=1}^{2}\left\|\sqrt{-\frac{\partial_t M_i}{M_i}}\partial_X|\nabla_L|MW^{+,2}_{\es}\right\|^2_{L^2H^N}\\
		&\qquad+\left|\int_{t_0}^{\et}\left\langle\partial_{XX}\partial_\sigma^{-1}T^{-t}_{2\alpha}M(W^{-,2}_{\es}+T^{-t}_{-2\alpha }W^{+,2}_{\es}),\partial_{XX}MW^{+,2}_{\es}\right\rangle_{H^N} dt\right|\\
		&\qquad+\left|\int_{t_0}^{\et}\left\langle\partial_{XY}^L\partial_\sigma^{-1}T^{-t}_{2\alpha}\nu\Delta_LMW^{-,2}_{\es},\partial_{XX}MW^{+,2}_{\es}\right\rangle_{H^N} dt+\mathcal{NL}\right|\Bigg)\\
		\leq&\ 3\epsilon^2/2+\mathcal{NL},
	\end{align*}
	where
	\begin{align*}
		\mathcal{NL}=& \left|\int_{t_0}^{\et}\left\langle\partial_{XY}^L\partial_\sigma^{-1}T^{-t}_{2\alpha}M((T^{-t}_{-2\alpha }W^+)\cdot\nabla_L W^{-,2})_{\es},\partial_{XX}MW^{+,2}_{\es}\right\rangle_{H^N} dt\right|\\
		&+\left|\int_{t_0}^{\et}\left\langle\partial_{XY}^L\partial_\sigma^{-1}T^{-t}_{2\alpha}M(\partial_Y^L\Delta_L^{-1}(\partial^L_j(T^{-t}_{-2\alpha }W^{+,i})\partial^L_iW^{-,j}))_{\es},\partial_{XX}MW^{+,2}_{\es}\right\rangle_{H^N} dt\right|.
	\end{align*}
	The estimate of $\mathcal{NL}$ can be considered similarly with the nonlinear terms below that arise in the energy estimate. The treatment of $OLS_4$ is exactly the same as the above. Hence we omit further details to conclude the estimate of $OLS$.
	
	Turning to nonlinear terms, for simplicity, we will drop the multiplier $T^{-t}_{2\alpha}$ and superscripts $``+,-"$ which are irrelevant in the nonlinearity.	For $NLT(j),j\in\{1,3\}$, using Lemma \ref{est1}, bootstrap assumptions, and integration by parts, we obtain
	\begin{align*}
		NLT(j)=&\int_{t_0}^{\et}\left\langle \partial_XM(W^j\partial^L_{j}W^{2})_{\es},|\nabla_L|\partial_X|\nabla_L|MW^{2}_{\es}\right\rangle_{H^N} dt\\
		\lesssim&(\|\partial_XMW^j_{\es}\|_{L^2H^N}+\|\partial_XMW^j_{\en}\|_{L^2H^N})\|\partial_{Xj}^LMW^2_{\es}\|_{L^\infty H^N}\|\nabla_L\partial_X|\nabla_L|W^2_{\es}\|_{L^2H^N}\\
		&+\|MW^j_0\|_{L^\infty H^N}\|\partial_{Xj}^LMW^2_{\es}\|_{L^2H^N}\|\nabla_L\partial_X|\nabla_L|MW^2_{\es}\|_{L^2H^N}\\
		&+(\|\partial_XMW^j_{\es}\|_{L^\infty H^N}+\|MW^j_{0}\|_{L^\infty H^N})\|\partial_{XX}MW^2_{\en}\|_{L^2 H^N}\|\nabla_L\partial_X|\nabla_L|MW^2_{\es}\|_{L^2H^N}\\
		&+\mathbf{1}_{j=3}(\|\partial_XMW^3_{\es}\|_{L^2H^N}+\|\partial_XMW^3_{\en}\|_{L^2H^N})\|\partial_{Z}MW^2_{0}\|_{L^\infty H^N}\|\nabla_L\partial_X|\nabla_L|MMW^2_{\es}\|_{L^2H^N}\\
		\lesssim&\ \nu^{-1}\epsilon^3.
	\end{align*} 
	Turning to $NLT(2,\cdot,\en)$, by bootstrap assumptions, it holds that
	\begin{align*}
		NLT(2,\cdot,\en)=&\int_{t_0}^{\et}\left\langle 	\partial_XM((W^2_{\es}+W^2_0)\partial_{Y}^LW^{2}_{\en}),|\nabla_L|\partial_X|\nabla_L|MW^{2}_{\es}\right\rangle_{H^N} dt\\
		\lesssim&(\|\partial_XMW^2_{\es}\|_{L^\infty H^N}+\|MW^2_0\|_{L^\infty H^N})\|\partial_{XY}^LMW^2_{\en}\|_{L^2H^N}\|\nabla_L\partial_X|\nabla_L|MW^2_{\es}\|_{L^2H^N}\\
		\lesssim&\ \nu^{-1}\epsilon^3.
	\end{align*}
	Considering $NLT(2,\en,0)$, we have
	\begin{align*}
		NLT(2,\en,0)=&\int_{t_0}^{\et}\left\langle |\nabla_L|M(\partial_XW^2_{\en}\partial_{Y}W^{2}_{0}),\partial_X|\nabla_L|MW^{2}_{\es}\right\rangle_{H^N} dt\\
		\lesssim&\|\partial_X|\nabla_L|MW^2_{\en}\|_{L^2H^N}\|\partial_YW^2_0\|_{L^2H^N}\|\partial_X|\nabla_L|MW^2_{\es}\|_{L^\infty H^N}\\
		&+\|\partial_XMW^2_{\en}\|_{L^2H^N}\|MF^2_0\|_{L^\infty H^N}\|\partial_X|\nabla_L|MW^2_{\es}\|_{L^2H^N}\\
		\lesssim&\ \nu^{-1}\epsilon^3+\nu^{-5/6}\epsilon^3.
	\end{align*}
	The treatment of the other contributions of $NLT(2)$ is similar to $NLT(3)$. We omit further details and complete the estimate of $NLT$. For $NLP$, we treat the components
	\begin{equation*}
		NLP(i,j)=\int_{t_0}^{\et}\left\langle \frac{\partial^L_{XY}}{|\nabla_L|}M(\partial^L_jW^{i}\partial^L_iW^{j})_{\es},\partial_X|\nabla_L|MW^{2}_{\es}\right\rangle_{H^N} dt
	\end{equation*}
	with $i,j\in\{1,2,3\}$ respectively. For $i=j=1$, integration by parts implies
	\begin{align*}
		NLP(1,1)		\lesssim&\|\partial_XMW^1_{\es}\|^2_{L^2H^N}\|\partial_{X}MW^1_{\es}\|_{L^\infty H^N}\|\partial_{XX}|\nabla_L|MW^2_{\es}\|_{L^2 H^N}\\
		&+\|\partial_XMW^1_{\es}\|^2_{L^2H^N}\|\partial_{X}MW^1_{\en}\|_{L^\infty H^N}\|\partial_{XX}|\nabla_L|MW^2_{\es}\|_{L^2 H^N}\\
		\lesssim&\ \nu^{-1}\epsilon^3.
	\end{align*}
	Turning to $i=1,j=2$, it holds that
	\begin{align*}
		NLP(1,2)\lesssim&(\|\partial_{XY}^LMW^1_{\neq}\|_{L^2H^N}+\|\partial_{Y}^LMW^1_{0}\|_{L^2H^N})\|\partial_{XX}MW^2_{\es}\|_{L^2H^N}\|\partial_{X}|\nabla_L|MW^2_{\es}\|_{L^\infty H^N}\\
		&+(\|\partial_{XY}^LMW^1_{\es}\|_{L^2H^N}+\|\partial_{Y}^LMW^1_{0}\|_{L^2H^N})\|\partial_{XX}MW^2_{\en}\|_{L^2H^N}\|\partial_{X}|\nabla_L|MW^2_{\es}\|_{L^\infty H^N}\\
		\lesssim&\ \nu^{-1}\epsilon^3.
	\end{align*}
	When $i=1,j=3$, noting $N>9/2$, we have
	\begin{align*}
		NLP(1,3,\es,\en)=& NLP^{HL}(1,3,\es,\en)+NLP^{LH}(1,3,\es,\en)\\
		\lesssim&\|\partial_ZMW^1_{\es}\|_{L^2H^N}\|\partial_XMW^3_{\en}\|_{L^\infty H^{3/2+}}\|\partial_{XX}|\nabla_L|MW^2_{\es}\|_{L^2H^N}\\
		&+\|\partial_ZMW^1_{\es}\|_{L^\infty  H^{3/2+}}\|\partial_XMW^3_{\en}\|_{L^2H^{N}}\|\partial_{XX}|\nabla_L|MW^2_{\es}\|_{L^2H^N}\\
		\lesssim&\ \nu^{-1}\epsilon^3.
	\end{align*}
	The other contributions can be treated similarly and we omit the details. For $i=j=2$, it follows that
	\begin{align*}
		NLP(2,2)\lesssim&(\|\partial_{XY}^LMW^2_{\neq}\|_{L^2H^N}+\|\partial_{Y}MW^2_{0}\|_{L^2H^N})\|\partial_{XY}^LMW^2_{\es}\|_{L^2H^N}\|\partial_X|\nabla_L|MW^2_{\es}\|_{L^\infty H^N}\\
		&+(\|\partial_{XY}^LMW^2_{\es}\|_{L^2H^N}+\|\partial_{Y}MW^2_{0}\|_{L^2H^N})\|\partial_{XY}^LMW^2_{\en}\|_{L^2H^N}\|\partial_X|\nabla_L|MW^2_{\es}\|_{L^\infty H^N}\\
		\lesssim&\ \nu^{-1}\epsilon^3.
	\end{align*}
	The treatment of $i=2,j=3$ is almost the same as $NLP(1,2)$, except $NLP(2,3,0,\neq)$. In fact, we have
	\begin{align*}
		NLP(2,3,0,\neq)\lesssim&\|\partial_{Z}MW^2_{0}\|_{L^\infty H^N}\|\partial_{XY}^LMW^3_{\neq}\|_{L^2H^N}\|\partial_X|\nabla_L|MW^2_{\es}\|_{L^2 H^N}\lesssim\ \nu^{-1}\epsilon^3.
	\end{align*}
	For $i=j=3$, by symmetry, we only need to control the $HL$ interaction
	\begin{align*}
		NLP^{HL}(3,3)\lesssim \|\partial_ZMW^3\|_{L^2H^N}\|\partial_ZMW^3\|_{L^\infty H^{3/2+}}\|\partial_{XX}|\nabla_L|MW^2_{\es}\|_{L^2H^N}\lesssim \nu^{-1/2}\epsilon\epsilon\nu^{-1/2}\epsilon=\nu^{-1}\epsilon^3,
	\end{align*}
	concluding this subsection.
	\subsection{Estimate of $F^{2}_{\es}$}
	In this subsection, we improve \eqref{bsf2}. An energy estimate gives
	\begin{align*}
		&\frac{1}{2}\|MF^{+,2}_{\es}(\et)\|^2_{H^N}+\nu\|\nabla_LMF^{+,2}_{\es}\|^2_{L^2H^N}+\sum_{i=1}^{2}\left\|\sqrt{-\frac{\partial_t M_i}{M_i}}MF^{+,2}_{\es}\right\|^2_{L^2H^N}\\
		&=\frac{1}{2}\|MF^{+,2}_{\es}(t_0)\|^2_{H^N}+\delta_0\nu^{1/3}\|MF^{+,2}_{\es}\|_{L^2H^N}^2+\int_{t_0}^{\et}\left\langle\partial^L_{XY}(T^{-t}_{2\alpha}MW^{-,2}_{\es}-MW^{+,2}_{\es}),MF^{+,2}_{\es}\right\rangle_{H^N} dt\\
		&\quad-\int_{t_0}^{\et}\left\langle M((T^{-t}_{2\alpha}W^{-})\cdot\nabla_LF^{+,2})_{\es},MF^{+,2}_{\es}\right\rangle_{H^N} dt-\int_{t_0}^{\et}\left\langle M((T^{-t}_{2\alpha}F^{-})\cdot\nabla_LW^{+,2})_{\es},MF^{+,2}_{\es}\right\rangle_{H^N} dt\\
		&\quad-2\int_{t_0}^{\et}\left\langle M(\partial^L_i(T^{-t}_{2\alpha}W^{-,j})\partial^L_{ij}W^{+,2})_{\es},MF^{+,2}_{\es}\right\rangle_{H^N} dt\\
		&\quad+\int_{t_0}^{\et}\left\langle \partial_Y^LM(\partial^L_j(T^{-t}_{2\alpha}W^{-,i})\partial^L_iW^{+,j})_{\es},MF^{+,2}_{\es}\right\rangle_{H^N} dt\\
		&=:\frac{1}{2}\|MF^{+,2}_{\es}(t_0)\|_{H^N}^2+L_{\lambda}+LS+NLT+NLS_1+NLS_2+NLP.
	\end{align*}
	The estimate of $L_\lambda$ is exactly the same as \eqref{estlambda}, and we omit it.
	For $LS$, using bootstrap assumption \eqref{bss} and $C_0>32$, it holds that 
	\begin{equation*}
		LS\leq\|\partial_X|\nabla_L|MW^2_{\es}\|_{L^2H^N}\|MF^2_{\es}\|_{L^2H^N}\leq64 \nu^{-1/6}\epsilon C_0\nu^{-1/2}\epsilon\leq 2(C_0\nu^{-1/3}\epsilon)^2.
	\end{equation*}
	Turning to $NLT(j)$, when $j=1$, we have
	\begin{align*}
		NLT(1)=&\int_{t_0}^{\et}\left\langle M(W^{1}\partial_XF^{2})_{\es},MF^{2}_{\es}\right\rangle_{H^N} dt\\
		\lesssim& \|MW^1\|_{L^\infty H^N}\|\partial_XMF^2_{\es}\|_{L^2H^N}\|MF^2_{\es}\|_{L^2H^N}\\
		&+(\|MW^1_{\es}\|_{L^\infty H^N}+\|MW^1_{0}\|_{L^\infty H^N})\|\partial_XMF^2_{\en}\|_{L^2H^N}\|MF^2_{\es}\|_{L^2H^N}\\
		\lesssim&\ \nu^{-1/3}\epsilon\nu^{-5/6}\epsilon\nu^{-1/2}\epsilon+\epsilon\nu^{-7/6}\epsilon\nu^{-1/2}\epsilon=\nu^{-1}\epsilon(\nu^{-1/3}\epsilon)^2.
	\end{align*}
	The estimate of $j=2,3$ is similar to the above, and we omit further details. For $NLS_1$, using divergence-free condition, we obtain 
	\begin{align*}
		NLS_1(1,\en,\es)=&\int_{t_0}^{\et}\langle M((-\partial_Y^L\partial_X^{-1}F^2_{\en}+\sigma F^3_{\en})\partial_XW^2_{\es}),MF^2_{\es}\rangle_{H^N}dt\\
		\lesssim&(\|\partial_Y^LMF^2_{\en}\|_{L^2H^N}+\|MF^3_{\en}\|_{L^2 H^N} )\|\partial_XMW^2_{\es}\|_{L^2H^N}\|MF^2_{\es}\|_{L^\infty H^N}\\
		\lesssim&\ \nu^{-7/6}\epsilon\epsilon\nu^{-1/3}\epsilon=\nu^{-5/6}\epsilon(\nu^{-1/3}\epsilon)^2,
	\end{align*}
	and 
	\begin{align*}
		NLS_1(1,\es,\en)=&NLS_1^{HL}(1,\es,\en)+NLS_1^{LH}(1,\es,\en)\\
		\lesssim&\|\partial_Y^LMF^2_{\es}\|_{L^2H^N}\|\partial_XMW^2_{\en}\|_{L^2H^N}\|MF^2_{\es}\|_{L^\infty H^N}\\
		&+\|\partial_ZMF^3_{\es}\|_{L^2H^{3/2+}}\|\partial_XMW^2_{\en}\|_{L^2H^N}\|MF^2_{\es}\|_{L^\infty H^N}\\
		&+\|\partial_ZMF^3_{\es}\|_{L^2H^{N}}\|\partial_XMW^2_{\en}\|_{L^2H^{3/2+}}\|MF^2_{\es}\|_{L^\infty H^N}\\
		\lesssim&\ 2\nu^{-5/6}\epsilon\nu^{-1/3}\epsilon\nu^{-1/3}\epsilon+\nu^{-7/6}\epsilon\nu^{-1/6}\epsilon\nu^{-1/3}\epsilon\lesssim\nu^{-1}\epsilon(\nu^{-1/3}\epsilon)^2.
	\end{align*}
	The other contributions of $NLS_1$, $NLS_2$, and $NLP$ can be estimated analogously and we omit further details. This completes the improvement of \eqref{bsf2}.
	\subsection{Estimate of $(\partial_X,\partial_Z)^2W^{3}_{\es}$}
	In this subsection, we improve \eqref{bsz3}. We only consider $\partial_{ZZ}W^3_{\es}$, and the estimate of $(\partial_X,\partial_Z)\partial_XW^3_{\es}$ is exactly the same. An energy estimate gives
	\begin{align*}
		&\frac{1}{2}\|\partial_{ZZ}MW^{+,3}_{\es}(\et)\|^2_{H^N}+\nu\|\nabla_L\partial_{ZZ}MW^{+,3}_{\es}\|^2_{L^2H^N}+\sum_{i=1}^{2}\left\|\sqrt{-\frac{\partial_t M_i}{M_i}}\partial_{ZZ}MW^{+,3}_{\es}\right\|^2_{L^2H^N}\\
		&=\frac{1}{2}\|\partial_{ZZ}MW^{+,3}_{\es}(t_0)\|^2_{H^N}+\delta\nu^{1/3}\int_{t_0}^{\et}\|\partial_{ZZ}MW^{+,3}_{\es}\|^2_{H^N}dt\\
		&\quad+\int_{t_0}^{\et}\left\langle \frac{\partial_{XZ}}{\Delta_L}\partial_{ZZ}(W^{-,2}_{\es}+T^{-t}_{-2\alpha }W ^{+,2}_{\es}),\partial_{ZZ}MW^{+,3}_{\es}\right\rangle_{H^N} dt\\
		&\quad-\int_{t_0}^{\et}\left\langle \partial_{ZZ}M((T^{-t}_{+2\alpha}W^{-})\cdot\nabla_LW^{+,3})_{\es},\partial_{ZZ}MW^{+,3}_{\es}\right\rangle_{H^N} dt\\
		&\quad+\int_{t_0}^{\et}\left\langle \frac{\partial_{ZZZ}}{\Delta_L}M(\partial^L_j(T^{-t}_{+2\alpha }W^{-,i})\partial^L_iW^{+,j})_{\es},\partial_{ZZ}MW^{+,3}_{\es}\right\rangle_{H^N} dt\\
		&=:\frac{1}{2}\|\partial_{ZZ}|\nabla_L|MW^{+,2}_{\es}(t_0)\|^2_{H^N}+L_{\lambda}+LP+NLT+NLP.
	\end{align*}
	For $LP$, it follows that 
	\begin{equation*}
		LP\leq\left\|\sqrt{-\frac{\partial_t M_1}{M_1}}\partial_{ZZ}MW^{2}_{\es}\right\|_{L^2H^N}\left\|\sqrt{-\frac{\partial_t M_1}{M_1}}\partial_{ZZ}MW^{3}_{\es}\right\|_{L^2H^N}\leq\frac{1}{C_0}(C_0\epsilon)^2.
	\end{equation*}
	Turing to nonlinear terms, we have
	\begin{align*}
		NLT^{LH}(1,\en,\es)=&\int_{t_0}^{\et}\left\langle \partial_{ZZ}M(W^{1,Lo}_{\en}\partial_XW^{3,Hi}_{\es}),\partial_{ZZ}MW^{3}_{\es}\right\rangle_{H^N} dt\\
		\lesssim&\|\partial_{XX}MW^1_{\en}\|_{L^\infty H^{3/2+}}\|\partial_{XZZ}MW^3_{\es}\|_{L^2H^N}\|\partial_{ZZ}MW^3_{\es}\|_{L^2H^N}\\
		\lesssim&\ \nu^{-1/3}\epsilon\nu^{-1/2}\epsilon\nu^{-1/6}\epsilon=\nu^{-1}\epsilon^3,
	\end{align*}
	and
	\begin{align*}
		NLT^{LH}(\cdot,\cdot,\en)=&\int_{t_0}^{\et}\left\langle \partial_{ZZ}M(W^{Lo}\cdot\nabla_LW^{3,Hi}_{\en})_{\es},\partial_{ZZ}MW^{3}_{\es}\right\rangle_{H^N} dt\\
		\lesssim&(\|MW_{\es}\|_{L^\infty H^N}+\|MW_{0}\|_{L^\infty H^N})\|\nabla_L\partial_{XX}MW^3_{\en}\|_{L^2H^N}\|\partial_{ZZ}MW^3_{\es}\|_{L^2H^N}\\
		\lesssim&\ \epsilon\nu^{-5/6}\epsilon\nu^{-1/6}\epsilon=\nu^{-1}\epsilon^3.
	\end{align*}
	The treatments of the other contributions of $NLT$ and $NLP$ are similar to the improvement of \eqref{bssymv2}, and we omit the details. This completes the improvement of \eqref{bsz3}.
	\subsection{Estimate of $F^{3}_{\es}$}
	In this subsection, we improve \eqref{bsf3}. An energy estimate gives
	\begin{align*}
		&\frac{1}{2}\|MF^{+,3}_{\es}(\et)\|^2_{H^N}+\nu\|\nabla_LMF^{+,3}_{\es}\|^2_{L^2H^N}+\sum_{i=1}^{2}\left\|\sqrt{-\frac{\partial_t M_i}{M_i}}MF^{+,3}_{\es}\right\|^2_{L^2H^N}\\
		&=\frac{1}{2}\|MF^{+,3}_{\es}(t_0)\|^2_{H^N}+\delta_0\nu^{1/3}\|MF^{+,3}_{\es}\|_{L^2H^N}^2-\int_{t_0}^{\et}\left\langle2\partial^L_{XY}MW^{+,3}_{\es},MF^{+,3}_{\es}\right\rangle_{H^N} dt\\
		&\quad+\int_{t_0}^{\et}\left\langle\frac{\partial_{XZ}}{\Delta_L}(F^{+,2}_{\es}+T^{-t}_{2\alpha}MF^{-,2}_{\es}),MF^{+,3}_{\es}\right\rangle_{H^N} dt-\int_{t_0}^{\et}\left\langle M((T^{-t}_{2\alpha}W^{-})\cdot\nabla_LF^{+,3})_{\es},MF^{+,3}_{\es}\right\rangle_{H^N} dt\\
		&\quad-\int_{t_0}^{\et}\left\langle M((T^{-t}_{2\alpha}F^{-})\cdot\nabla_LW^{+,3})_{\es},MF^{+,3}_{\es}\right\rangle_{H^N} dt-2\int_{t_0}^{\et}\left\langle M(\partial^L_i(T^{-t}_{2\alpha}W^{-,j})\partial^L_{ij}W^{+,3})_{\es},MF^{+,2}_{\es}\right\rangle_{H^N} dt\\
		&\quad+\int_{t_0}^{\et}\left\langle \partial_ZM(\partial^L_j(T^{-t}_{2\alpha}W^{-,i})\partial^L_iW^{+,j})_{\es},MF^{+,3}_{\es}\right\rangle_{H^N} dt\\
		&=:\frac{1}{2}\|MF^{+,3}_{\es}(t_0)\|_{H^N}^2+L_{\lambda}+LS+LP+NLT+NLS_1+NLS_2+NLP.	
	\end{align*}
	For $LS$, by bootstrap assumption \eqref{bsz3}, we have
	\begin{equation*}
		LS\leq2\|\nabla_L\partial_XMW^3_{\es}\|_{L^2H^N}\|MF^3_{\es}\|_{L^2H^N}\leq 2C_0\nu^{-1/2}\epsilon C_0^2\nu^{-5/6}\epsilon=\frac{2}{C_0}(C_0^2\nu^{-2/3}\epsilon)^2.
	\end{equation*}
	Turning to $LP$, it holds that 
	\begin{equation*}
		LP\leq\left\|\sqrt{-\frac{\partial_t M_1}{M_1}}MF^{2}_{\es}\right\|_{L^2H^N}\left\|\sqrt{-\frac{\partial_t M_1}{M_1}}MF^{3}_{\es}\right\|_{L^2H^N}\leq\frac{1}{C_0}(C_0^2\nu^{-2/3}\epsilon)^2.
	\end{equation*}
	Regarding the nonlinear terms, for $NLT(\cdot,\cdot,\en)$, we obtain
	\begin{align*}
		NLT(\cdot,\cdot,\en)=&\int_{t_0}^{\et}\langle M(W\cdot\nabla_LF^3_{\en}),MF^3_{\es}\rangle_{H^N}dt\\
		\lesssim&(\|MW_{\es}\|_{L^\infty H^N}+\|MW_0\|_{L^\infty H^N})\|\nabla_LMF^3_{\en}\|_{L^2 H^N}\|MF^3_{\es}\|_{L^2 H^N}\\
		\lesssim&\ \nu^{-1}\epsilon(\nu^{-2/3}\epsilon)^2.
	\end{align*}
	Considering $NLT(1,\en,\es)$, by Lemma \ref{est1}, it holds that
	\begin{align*}
		NLT(1,\en,\es)=&\int_{t_0}^{\et}\langle M(W^1_{\en}\partial_XF^3_{\es}),MF^3_{\es}\rangle_{H^N}dt\\
		\lesssim&\||\nabla_L|MW^1_{\en}\|_{L^2 H^N}\||\nabla_L|\partial_XMW^3_{\es}\|_{L^2H^N}\|MF^3_{\es}\|_{L^\infty H^N}\\
		&+\|MW^1_{\en}\|_{L^\infty H^N}\||\nabla_L|\partial_XMW^3_{\es}\|_{L^2H^N}\||\nabla_L|MF^3_{\es}\|_{L^2 H^N}\\
		\lesssim&\ \nu^{-2/3}\epsilon(\nu^{-2/3}\epsilon)^2.
	\end{align*}
	The estimates of the other contribution of $NLT$, $NLS_1$, $NLS_2$, and $NLP$ are simliar to the above and we omit the details. This completes the improvement of \eqref{bsf3}.
	\section{Energy estimates on \nm}
	In this section, we improve the bootstrap assumptions \eqref{bsnu} and \eqref{bsnb}. We focus on the nonlinear interaction between two homogeneous modes and the energy estimate on low regularity of $Q^2_{\en}$ \eqref{bsq2l}. In fact, the interactions between two non-homogeneous modes, or zero-mode and nonzero modes are similarly treated as Section \ref{secs}. 
	\subsection{Estimate of $Q^2_{\en}$ in $H^N$}
	In this subsection, we improve \eqref{bsq2}. Recalling the system of $Q^2_{\en}$ \eqref{eqqh}, an energy estimate gives
	\begin{align*}
		&\frac{1}{2}\|MQ^2_{\en}(\et)\|^2_{H^N}+\nu\|\nabla_LMQ^2_{\en}\|_{L^2H^N}^2+\sum_{i=1}^{2}\left\|\sqrt{-\frac{\partial_t M_i}{M_i}}MQ^{2}_{\en}\right\|^2_{L^2H^N}\\
		&=\frac{1}{2}\|MQ^{2}_{\en}(t_0)\|^2_{H^N}+\delta_0\nu^{1/3}\|MQ^2_{\en}\|_{L^2H^N}^2-\int_{t_0}^{\et}\langle M(U\cdot \nabla_LQ^2)_{\en},MQ^2_{\en}\rangle_{H^N}dt\\
		&\quad-\int_{t_0}^{\et}\langle M(Q\cdot \nabla_LU^2)_{\en},MQ^2_{\en}\rangle_{H^N}dt-2\int_{t_0}^{\et}\langle M(\partial_i^LU^j\partial_{ij}^LU^2)_{\en},MQ^2_{\en}\rangle_{H^N}dt\\
		&\quad+\int_{t_0}^{\et}\langle \Delta_LM(B\cdot \nabla_LB^2)_{\en},MQ^2_{\en}\rangle_{H^N}dt+\int_{t_0}^{\et}\langle \partial_Y^LM(\partial^L_jU^i\partial_i^LU^j-\partial^L_jB^i\partial^L_iB^j)_{\en},MQ^2_{\en}\rangle_{H^N}dt\\
		&=:\frac{1}{2}\|MQ^{2}_{\es}(t_0)\|^2_{H^N}+L_{\lambda}+NLT+NLS_1+NLS_2+NLF+NLP.
	\end{align*} 
	The estimates of $NLT,NLS_1$, and $NLS_2$ are standard, and we omit the details. For $NLF$, by integration by parts, it follows that
	\begin{align*}
		NLF(\cdot,\en,\en)=&-\int_{t_0}^{\et}\left\langle |\nabla_L|M(B_{\en}\cdot \nabla_LB^2_{\en}),|\nabla_L|MQ^2_{\en}\right\rangle_{H^N}dt\\
		\lesssim&\||\nabla_L|(MB^1_{\en},MB^3_{\en})\|_{L^2H^N}\|\partial_XMB^2_{\en}\|_{L^\infty H^N}\||\nabla_L|MQ^2_{\en}\|_{L^2H^N}\\
		&+\|(MB^1_{\en},MB^3_{\en})\|_{L^\infty H^N}\||\nabla_L|\partial_XMB^2_{\en}\|_{L^2H^N}\||\nabla_L|MQ^2_{\en}\|_{L^2H^N}\\
		&+\||\nabla_L|MB^2_{\en}\|_{L^2H^N}\|\partial_Y^LMB^2_{\en}\|_{L^\infty H^N}\||\nabla_L|MQ^2_{\en}\|_{L^2H^N}\\
		&+\|MB^2_{\en}\|_{L^\infty H^N}\||\nabla_L|\partial_Y^LMB^2_{\en}\|_{L^2H^N}\||\nabla_L|MQ^2_{\en}\|_{L^2H^N}\\
		\lesssim&\ (\nu^{-5/6}\epsilon\epsilon+\nu^{-1/3}\epsilon\nu^{-1/2}\epsilon+\nu^{-1/2}\epsilon\nu^{-1/3}\epsilon+\epsilon\nu^{-5/6}\epsilon)\nu^{-5/6}\epsilon\lesssim\nu^{-1}\epsilon(\nu^{-1/3}\epsilon)^2.
	\end{align*}
	The treatment of $NLP$ is exactly the same. Hence we omit further details to conclude the improvement of \eqref{bsq2}.
	\subsection{Estimate of $Q^2_{\en}$ in $H^{N-1}$}
	Here we consider \eqref{bsq2l} which implies the inviscid damping of $U^2_{\en}$. An energy estimate gives
	\begin{align*}
		&\frac{1}{2}\|\langle \et\rangle^{-1}MQ^2_{\en}(\et)\|^2_{H^{N-1}}+\nu\left\|\langle t\rangle^{-1}\nabla_LMQ^2_{\en}\right\|_{L^2H^{N-1}}^2+\sum_{i=1}^{2}\left\|\sqrt{-\frac{\partial_t M_i}{M_i}}\langle t\rangle^{-1}MQ^{2}_{\es}\right\|^2_{L^2H^{N-1}}\\
		&=\frac{1}{2}\|\langle t_0\rangle^{-1}MQ^{2}_{\en}(t_0)\|^2_{H^{N-1}}-\left\|\sqrt{t}\langle t\rangle^{-2} MQ^2_{\en}\right\|^2_{L^2H^{N-1}}+\delta_0\nu^{1/3}\|\langle t\rangle^{-1}MQ^2_{\en}\|_{L^2H^{N-1}}^2
		\\
		&\quad-\int_{t_0}^{\et}\left\langle \langle t\rangle^{-1}M(Q\cdot \nabla_LU^2)_{\en},\langle t\rangle^{-1}MQ^2_{\en}\right\rangle_{H^{N-1}}dt-2\int_{t_0}^{\et}\left\langle \langle t\rangle^{-1}M(\partial_i^LU^j\partial_{ij}^LU^2)_{\en},\langle t\rangle^{-1}MQ^2_{\en}\right\rangle_{H^{N-1}}dt\\
		&\quad-\int_{t_0}^{\et}\left\langle \langle t\rangle^{-1}M(U\cdot \nabla_LQ^2)_{\en},\langle t\rangle^{-1}MQ^2_{\en}\right\rangle_{H^{N-1}}dt+\int_{t_0}^{\et}\left\langle \langle t\rangle^{-1}\Delta_LM(B\cdot \nabla_LB^2)_{\en},\langle t\rangle^{-1}MQ^2_{\en}\right\rangle_{H^{N-1}}dt\\
		&\quad+\int_{t_0}^{\et}\left\langle \langle t\rangle^{-1}\partial_Y^LM(\partial^L_jU^i\partial_i^LU^j-\partial^L_jB^i\partial^L_iB^j)_{\en},\langle t\rangle^{-1}MQ^2_{\en}\right\rangle_{H^{N-1}}dt\\
		&=:\frac{1}{2}\langle t\rangle^{-2}\|MQ^{2}_{\en}(t_0)\|^2_{H^{N-1}}-\left\|\sqrt{t}\langle t\rangle^{-2} MQ^2_{\en}\right\|^2_{L^2H^{N-1}}+L_{\lambda}+NLS_1+NLS_2+NLT+NLF+NLP.
	\end{align*}
	For $NLS_1$, we start with the interaction between two non-homogeneous modes. Noting that
	\begin{align*}
		&\|Q^1_{\es}\|_{L^\infty H^{N-1}}\lesssim\nu^{-1/3}\|e^{\delta_0\nu^{1/3}t}Q^2_{\es}\|_{L^\infty H^N}+\|Q^3_{\es}\|_{L^\infty H^N}\lesssim\nu^{-2/3}\epsilon,\\
		&\|\langle t\rangle^{-1}\partial^L_YMU^2_{\es}\|_{L^2H^{N-1}}\lesssim\|MU^2_{\es}\|_{L^2H^N}\lesssim\epsilon,
	\end{align*}
	we have
	\begin{align*}
		NLS_1(\es,\es)=&\int_{t_0}^{\et}\left\langle \langle t\rangle^{-1}M(Q_{\es}\cdot \nabla_LU^2_{\es}),\langle t\rangle^{-1}MQ^2_{\en}\right\rangle_{H^{N-1}}dt\\
		\lesssim&\|Q^1_{\es}\|_{L^\infty H^{N-1}}\|\langle t\rangle^{-1}\|_{L^2_t}\|MU^2_{\es}\|_{L^2 H^N}\|\langle t\rangle^{-1}MQ^2_{\en}\|_{L^\infty H^{N-1}}\\
		&+\|MQ^2_{\es}\|_{L^2H^N}\|\langle t\rangle^{-1}\partial_Y^LMU^2_{\es}\|_{L^2H^N}\|\langle t\rangle^{-1}MQ^2_{\en}\|_{L^\infty H^{N-1}}\\
		&+\|MQ^3_{\es}\|_{L^\infty H^{N-1}}\|\langle t\rangle^{-1}\|_{L^2_t}\|MU^2_{\es}\|_{L^2 H^N}\|\langle t\rangle^{-1}MQ^2_{\en}\|_{L^\infty H^{N-1}}\\
		\lesssim&\ \nu^{-2/3}\epsilon^3+\nu^{-1/2}\epsilon^3\lesssim\nu^{-1}\epsilon^3.
	\end{align*}
	Turning next to the interaction between two homogeneous modes
	\begin{equation*}
		NLS_1(\en,\en)=\int_{t_0}^{\et}\left\langle \langle t\rangle^{-1}M(Q_{\en}\cdot \nabla_LU^{2}_{\en}),\langle t\rangle^{-1}MQ^2_{\en}\right\rangle_{H^{N-1}}dt.
	\end{equation*}
	By divergence-free condition and \eqref{estu2l}, we obtain
	\begin{align*}
		NLS_1(\en,\en)=&NLS_1^{HL}(\en,\en)+	NLS_1^{LH}(\en,\en)\\
		\lesssim&\|\partial_Y^L\langle t\rangle^{-1}MQ^2_{\en}\|_{L^2H^{N-1}}\|MU^2_{\en}\|_{L^2H^{N}}\|\langle t\rangle^{-1}MQ^2_{\en}\|_{L^\infty H^{N-1}}\\
		&+\|MQ^3_{\en}\|_{L^\infty H^{N-1}}\|\partial_XMU^2_{\en}\|_{L^2H^{3/2+}}\|\langle t\rangle^{-1}\|_{L^2_t}\|\langle t\rangle^{-1}MQ^2_{\en}\|_{L^\infty H^{N-1}}\\
		&+\|MQ^3_{\en}\|_{L^\infty H^{3/2+}}\|MU^2_{\en}\|_{L^2H^N}\|\langle t\rangle^{-1}\|_{L^2_t}\|\langle t\rangle^{-1}MQ^2_{\en}\|_{L^\infty H^{N-1}}\\
		&+\|\langle t\rangle^{-1}MQ^2_{\en}\|_{L^2H^{N-1}}\|\partial_{Y}^LMU^2_{\en}\|_{L^2H^{N-1}}\|\langle t\rangle^{-1}MQ^2_{\en}\|_{L^\infty H^{N-1}}\\
		\lesssim& \ \nu^{-1/2}\epsilon\nu^{-1/3}\epsilon^2+\nu^{-1}\epsilon^3+\nu^{-2/3}\epsilon\nu^{-1/3}\epsilon^2+\nu^{-1/6}\epsilon\nu^{-1/3}\epsilon^2\lesssim\nu^{-1}\epsilon^3.
	\end{align*}
	Finally, for the interaction between zero-mode and nonzero modes, noting $\mathbb{P}_{l=0}U^2_0=0$, we have
	\begin{align*}
		NLS_1(\neq ,0)+NLS_1(0,\neq)\lesssim&(\|MQ^2_{\es}\|_{L^2 H^N}+\|MQ^3_{\es}\|_{L^2 H^N})\|\mathbb{P}_{l\neq 0}MU^2_{0}\|_{L^\infty H^N}\|\langle t\rangle^{-1}MQ^2_{\en}\|_{L^2H^{N-1}}\\
		&+\|(MQ^1_0,MQ^3_0)\|_{L^\infty H^N}\|MU^2_{\neq }\|_{L^2H^N}\|\langle t\rangle^{-1}\|_{L^2_t}\|\langle t\rangle^{-1}MQ^2_{\en}\|_{L^\infty H^{N-1}}\\
		&+\|\mathbb{P}_{l\neq 0}MQ^2_{0}\|_{L^\infty H^N}\|\partial^L_YMU^2_{\es}\|_{L^2H^N}\|\langle t\rangle^{-1}MQ^2_{\en}\|_{L^2 H^{N-1}}\\
		\lesssim&\ \nu^{-5/6}\epsilon^2\nu^{-1/6}\epsilon+\nu^{-2/3}\epsilon\nu^{-1/3}\epsilon^2+\nu^{-1/3}\epsilon\nu^{-1/6}\epsilon^2\lesssim\nu^{-1}\epsilon^3,
	\end{align*}
	completing the estimate of $NLS_1$. The treatments of $NLS_2$ and $NLT$ are similar to the above term, and we omit further details. Turning to $NLF$, using integration by parts, it holds that
	\begin{align*}
		NLF=&-\int_{t_0}^{\et}\left\langle \langle t\rangle^{-1}|\nabla_L|M(B\cdot \nabla_LB^2)_{\en},|\nabla_L|\langle t\rangle^{-1}MQ^2_{\en}\right\rangle_{H^{N-1}}dt\\
		\lesssim&\|MB^1\|_{L^\infty H^N}\|\partial_XMB^2_{\neq}\|_{L^2H^N}\|\nabla_L\langle t\rangle^{-1}MQ^2_{\en}\|_{L^2H^{N-1}}\\
		&+\|MB^2\|_{L^\infty H^N}\|\partial_Y^LMB^2\|_{L^2H^N}\|\nabla_L\langle t\rangle^{-1}MQ^2_{\en}\|_{L^2H^{N-1}}\\
		&+(\|MB^3_{\es}\|_{L^\infty H^N}+\|MB^3_{0}\|_{L^\infty H^N})\|\partial_ZMB^2\|_{L^2 H^N}\|\nabla_L\langle t\rangle^{-1}MQ^2_{\en}\|_{L^2H^{N-1}}\\
		&+\|MB^3_{\en}\|_{L^\infty H^N}\|\partial_XMB^2_{\en}\|_{L^2 H^N}\|\nabla_L\langle t\rangle^{-1}MQ^2_{\en}\|_{L^2H^{N-1}}\\
		\lesssim&\ \nu^{-1/3}\epsilon\nu^{-1/6}\epsilon\nu^{-1/2}\epsilon+2\epsilon\nu^{-1/2}\epsilon\nu^{-1/2}\epsilon+\nu^{-1/3}\epsilon\nu^{-1/6}\epsilon\nu^{-1/2}\epsilon\lesssim\nu^{-1}\epsilon^3.
	\end{align*}
	The $NLP$ term can be treated similarly and is hence omitted for the sake of brevity. This completes the improvement of \eqref{bsq2l}.
	\subsection{Estimates of $\partial_{XX}U^3_{\en}$ and $Q^3_{\en}$} 
	In this subsection we improve \eqref{bsu3} and \eqref{bsq3}. For $\partial_{XX}U^3_{\en}$ and $N'\in\{N,N-2\}$, an energy estimate gives
	\begin{align*}
		&\frac{1}{2}\|\partial_{XX}MU^3_{\en}(\et)\|^2_{H^{N'}}+\nu\|\nabla_L\partial_{XX}MU^3_{\en}\|_{L^2H^{N'}}^2+\sum_{i=1}^{2}\left\|\sqrt{-\frac{\partial_t M_i}{M_i}}\partial_{XX}MU^{3}_{\en}\right\|^2_{L^2H^{N'}}\\
		&=\frac{1}{2}\|\partial_{XX}MU^3_{\en}(t_0)\|^2_{H^{N'}}+\delta_0\nu^{1/3}\|\partial_{XX}MU^3_{\en}\|_{L^2H^{N'}}^2-\sigma\int_{t_0}^{\et}\left\langle \partial_{XX}\Delta_L^{-1}\partial_{XX}MU^2_{\en},\partial_{XX}MU^3_{\en}\right\rangle_{H^{N'}}dt\\
		&\quad-\int_{t_0}^{\et}\left\langle \partial_{XX}M(U\cdot\nabla_LU^3)_{\en},\partial_{XX}MU^3_{\en}\right\rangle_{H^{N'}}dt+\int_{t_0}^{\et}\left\langle \partial_{XX}M(B\cdot\nabla_LB^3)_{\en},\partial_{XX}MU^3_{\en}\right\rangle_{H^{N'}}dt\\
		&\quad-\sigma\int_{t_0}^{\et}\left\langle \partial_{XXX}\Delta_L^{-1}M(\partial^L_jU^i\partial^L_iU^j-\partial^L_jB^i\partial^L_iB^j)_{\en},\partial_{XX}MU^3_{\en}\right\rangle_{H^{N'}}dt\\
		&=:\frac{1}{2}\|\partial_{XX}MU^3_{\en}(t_0)\|^2_{H^{N'}}+\delta_0\nu^{1/3}\|\partial_{XX}MU^3_{\en}\|_{L^2H^{N'}}^2+L_\lambda+LP+NLT+NLF+NLP.
	\end{align*}
	When $N'=N$, \eqref{estu2} yields
	\begin{align*}
		LP\leq\sigma\|\partial_X|\nabla_L|^{-1}\partial_{XX}MU^2_{\en}\|_{L^2H^N}\|\partial_X|\nabla_L|^{-1}\partial_{XX}MU^3_{\en}\|_{L^2H^N}\leq\frac{\sigma}{C_0}(C_0\nu^{-1/3}\epsilon)^2.
	\end{align*}
	The treatment of nonlinear terms are similar to that of \eqref{bsz3}. For instance, the $NLF(\cdot,\en,\en)$ term follows
	\begin{align*}
		NLF(\cdot,\en,\en)\lesssim&\int_{t_0}^{\et}\left\langle \partial_{XX}M(B^{Hi}_{\en}\cdot\nabla_LB^{3,Lo}_{\en}),\partial_{XX}MU^3_{\en}\right\rangle_{H^{N}}dt\\
		&+\int_{t_0}^{\et}\left\langle \partial_{XX}M(B^{Lo}_{\en}\cdot\nabla_LB^{3,Hi}_{\en}),\partial_{XX}MU^3_{\en}\right\rangle_{H^{N}}dt\\
		\lesssim&\|\partial_{XX}(MB^1_{\en},MB^3_{\en})\|_{L^2H^N}
		\|\partial_XMB^3_{\en}\|_{L^2H^{3/2+}}\|\partial_{XX}MU^3_{\en}\|_{L^\infty H^N}\\
		&+\|\partial_{XX}MB^2_{\en}\|_{L^2H^N}
		\|\partial_Y^LMB^3_{\en}\|_{L^2H^{3/2+}}\|\partial_{XX}MU^3_{\en}\|_{L^\infty H^N}\\
		&+\|MB_{\en}\|_{L^\infty H^{3/2+}}\|\nabla_L\partial_{XX}MU^3_{\en}\|_{L^2H^N}\|\partial_{XX}MU^3_{\en}\|_{L^2H^N}\\
		\lesssim&\ \nu^{-1/2}\epsilon\nu^{-1/6}\epsilon^2+\nu^{-1/3}\epsilon\nu^{-1/2}\epsilon^2+\nu^{-1/3}\epsilon\nu^{-5/6}\epsilon\nu^{-1/2}\epsilon\lesssim\nu^{-1}\epsilon(\nu^{-1/3}\epsilon)^2.
	\end{align*}
	When $N'=N-2$, using \eqref{estu2l}, we have
	\begin{align*}
		LP\leq\sigma\|\partial_X|\nabla_L|^{-1}\partial_{XX}MU^2_{\en}\|_{L^2H^{N-2}}\|\partial_X|\nabla_L|^{-1}\partial_{XX}MU^3_{\en}\|_{L^2H^{N-2}}\leq\frac{\sigma}{C_0}(C_0\epsilon)^2.
	\end{align*}
	For $NLF(\cdot,\en,\en)$, we have
	\begin{align*}
		NLF(\cdot,\en,\en)\lesssim&\|(MB^1_{\en},MB^3_{\en})\|_{L^\infty H^{N}}\|\partial_XMB^3_{\en}\|_{L^2H^{N}}\|\partial_XMU^3_{\en}\|_{L^2H^{N-2}}\\
		&+\|MB^2_{\en}\|_{L^\infty H^{N}}\|\partial^L_YMB^3_{\en}\|_{L^2H^{N}}\|\partial_XMU^3_{\en}\|_{L^2H^{N-2}}\\
		\lesssim&\ \nu^{-2/3}\epsilon^3.
	\end{align*}
	For the sake of brevity, we omit further details about the other estimates of nonlinear terms, and complete the improvement of \eqref{bsu3}. 
	
	Turning to $Q^3_{\en}$, for $N'\in\{N,N-2\}$, the energy estimate reads 
		\begin{align*}
			&\frac{1}{2}\|MQ^{3}_{\en}(\et)\|^2_{H^{N'}}+\nu\|\nabla_LMQ^{3}_{\en}\|^2_{L^2H^{N'}}+\sum_{i=1}^{2}\left\|\sqrt{-\frac{\partial_t M_i}{M_i}}MQ^{3}_{\en}\right\|^2_{L^2H^{N'}}\\
			&=\frac{1}{2}\|MQ^{3}_{\en}(t_0)\|^2_{H^{N'}}+\delta_0\nu^{1/3}\|MQ^{3}_{\en}\|_{L^2H^{N'}}^2-\int_{t_0}^{\et}\left\langle2\partial^L_{XY}MU^{3}_{\en},MQ^{3}_{\en}\right\rangle_{H^{N'}} dt\\
			&\quad+2\sigma\int_{t_0}^{\et}\left\langle\partial_{XX}\Delta^{-1}_LQ^{2}_{\en},MQ^{3}_{\en}\right\rangle_{H^{N'}} dt-\int_{t_0}^{\et}\left\langle \Delta_LM(U\cdot\nabla_LU^{3})_{\en},MQ^{3}_{\en}\right\rangle_{H^{N'}} dt\\
			&\quad+\int_{t_0}^{\et}\left\langle \Delta_LM(B\cdot\nabla_LB^{3})_{\en},MQ^{3}_{\en}\right\rangle_{H^{N'}} dt-\sigma\int_{t_0}^{\et}\left\langle \partial_X(\partial^L_jU^i\partial_i^LU^j-\partial^L_jB^i\partial^L_iB^j)_{\en},MQ^{3}_{\en}\right\rangle_{H^{N'}} dt.
		\end{align*}
		The main amplification arises from the linear stretching term. In fact, it holds that
		\begin{equation*}
			\int_{t_0}^{\et}\left\langle2\partial^L_{XY}MU^{3}_{\en},MQ^{3}_{\en}\right\rangle_{H^{N}} dt\leq2\|\partial_Y^L\partial_XMU^3_{\en}\|_{L^2H^N}\|MQ^3_{\en}\|_{L^2H^N}\leq\frac{2}{C_0}(C_0^2\nu^{-1}\epsilon)^2,
		\end{equation*}
		and 
		\begin{equation*}
			\int_{t_0}^{\et}\left\langle2\partial^L_{XY}MU^{3}_{\en},MQ^{3}_{\en}\right\rangle_{H^{N-2}} dt\leq2\|\partial_Y^L\partial_XMU^3_{\en}\|_{L^2H^{N-2}}\|MQ^3_{\en}\|_{L^2H^{N-2}}\leq\frac{2}{C_0}(C_0^2\nu^{-2/3}\epsilon)^2.
		\end{equation*}
		The estimates of nonlinear terms are treated similarly to \eqref{bsf3} and hence are omitted for brevity. This completes the improvement of \eqref{bsq3}. 
		\subsection{Estimates of $\nabla_LB^2_{\en}$ and $\nabla_L^2B^2_{\en}$}
		In this subsection, we improve \eqref{bsb2} and \eqref{bsh2}. 
		While the presence of the pressure term restricts our estimates to the quantity $Q^2_{\en}$, its absence enables us to deal with $\nabla_LB^2_{\en}$ and $\nabla_L^2B^2_{\en}$. We only show the treatment of $\partial_Y^LMB^2_{\en}$ which helps us to deduce the amplification of $B^1_{\en}$ and the treatment of $\partial_{XX}MB^2_{\en}$ which has an amplification of order $\nu^{-1/6}$. The other bootstrap arguments can be treated similarly to \eqref{bss}. For $\partial_Y^LMB^2_{\en}$, an energy estimate gives
		\begin{align*}
			&\frac{1}{2}\|\partial_Y^LMB^2_{\en}(\et)\|^2_{H^{N}}+\nu\|\nabla_L\partial_Y^LMB^2_{\en}\|_{L^2H^{N}}^2+\sum_{i=1}^{2}\left\|\sqrt{-\frac{\partial_t M_i}{M_i}}\partial_Y^LMB^{2}_{\en}\right\|^2_{L^2H^{N}}\\
			&=\frac{1}{2}\|\partial_Y^LMB^2_{\en}(t_0)\|^2_{H^{N}}+\delta_0\nu^{1/3}\|\partial_Y^LMB^2_{\en}\|_{L^2H^{N}}^2-\int_{t_0}^{\et}\left\langle \partial_{X}MB^2_{\en},\partial_Y^LMB^2_{\en}\right\rangle_{H^{N}}dt\\
			&\quad-\int_{t_0}^{\et}\left\langle \partial_Y^LM(U\cdot\nabla_LB^2)_{\en},\partial_Y^LMB^2_{\en}\right\rangle_{H^{N}}dt+\int_{t_0}^{\et}\left\langle \partial_Y^LM(B\cdot\nabla_LU^2)_{\en},\partial_Y^LMB^2_{\en}\right\rangle_{H^{N}}dt\\
			&=:\frac{1}{2}\|\partial_XMU^3_{\en}(t_0)\|^2_{H^{N'}}+\delta_0\nu^{1/3}\|\partial_XMU^3_{\en}\|_{L^2H^{N'}}^2+L_\lambda+LS+NLT+NLS.
		\end{align*}
		By the enhanced dissipation in \eqref{bsb2}, we obtain
		\begin{equation*}
			LS\leq\|\partial_XMB^2_{\en}\|_{L^2H^N}\|\partial_Y^LMB^2_{\en}\|_{L^2H^N}\leq\frac{1}{C_0}(C_0\nu^{-1/3}\epsilon)^2.
		\end{equation*}
	 	Turning to the nonlinear terms, using integration by parts, it holds that
		\begin{align*}
			NLT(\en,\en)+NLS(\en,\en)\lesssim&\|MU_{\en}\|_{L^\infty H^N}\||\nabla_L|MB^2_{\en}\|_{L^2H^N}\|\partial_Y^L\partial_Y^LMB^2_{\en}\|_{L^2H^N}\\
			&+\|MB_{\en}\|_{L^\infty H^N}\||\nabla_L|MU^2_{\en}\|_{L^2H^N}\|\partial_Y^L\partial_Y^LMB^2_{\en}\|_{L^2H^N}\\
			\lesssim&\ \nu^{-1/3}\epsilon\nu^{-1/2}\epsilon\nu^{-5/6}\epsilon+\nu^{-1/3}\epsilon\nu^{-1/3}\epsilon\nu^{-5/6}\epsilon\lesssim\nu^{-1}\epsilon(\nu^{-1/3}\epsilon)^2.
		\end{align*}
		This completes the estimate of $\partial_Y^LMB^2_{\en}$. For $\partial_{XX}MB^2_{\en}$, an energy estimate gives
		\begin{align*}
			&\frac{1}{2}\|\partial_{XX}MB^2_{\en}(\et)\|^2_{H^{N}}+\nu\|\nabla_L\partial_{XX}MB^2_{\en}\|_{L^2H^{N}}^2+\sum_{i=1}^{2}\left\|\sqrt{-\frac{\partial_t M_i}{M_i}}\partial_{XX}MB^{2}_{\en}\right\|^2_{L^2H^{N}}\\
			&=\frac{1}{2}\|\partial_{XX}MB^2_{\en}(t_0)\|^2_{H^{N}}+\delta_0\nu^{1/3}\|\partial_{XX}MB^2_{\en}\|_{L^2H^{N}}^2
			-\int_{t_0}^{\et}\left\langle \partial_{XX}M(U\cdot\nabla_LB^2)_{\en},\partial_{XX}MB^2_{\en}\right\rangle_{H^{N}}dt\\
			&\quad+\int_{t_0}^{\et}\left\langle \partial_{XX}M(B\cdot\nabla_LU^2)_{\en},\partial_{XX}MB^2_{\en}\right\rangle_{H^{N}}dt.
		\end{align*}
		Considering the nonlinear term $\partial_{XX}(B^{1,Lo}_{\en}\partial_XU^{2,Hi}_{\en})$, by \eqref{estu2} and integration by parts, we have
		\begin{align*}
			\int_{t_0}^{\et}\left\langle \partial_{XX}M(B^{1,Lo}_{\en}\partial_XU^{2,Hi}_{\en}),\partial_{XX}MB^2_{\en}\right\rangle_{H^{N}}dt\lesssim&\|MB^1_{\en}\|_{L^\infty H^N}\|\partial_{XX}U^2_{\en}\|_{L^2H^N}\|\partial_X\partial_{XX}MB^2_{\en}\|_{L^2H^N}\\
			\lesssim&\ \nu^{-1/3}\epsilon\nu^{-1/3}\epsilon\nu^{-2/3}\epsilon=\nu^{-1}\epsilon(\nu^{-1/6}\epsilon)^2.
		\end{align*}
		The other nonlinear terms can be treated similarly and are hence omitted. This completes the improvement of \eqref{bsb2}--\eqref{bsh2}.
	\subsection{Estimates of $\partial_{X}B^3_{\en},\partial_{XX}B^3_{\en}$, and $G^3_{\en}$}
	In this subsection, we improve \eqref{bsb3} and \eqref{bsh3}. Unlike $U^3_{\en}$, the amplification of $B^3_{\en}$ is caused by the nonlinear interaction. In fact, for $s\in\{1,2\}$, the energy estimate reads 
	\begin{align*}
		&\frac{1}{2}\|\partial_X^sMB^3_{\en}(\et)\|^2_{H^{N}}+\nu\|\nabla_L\partial_X^sMB^3_{\en}\|_{L^2H^{N}}^2+\sum_{i=1}^{2}\left\|\sqrt{-\frac{\partial_t M_i}{M_i}}\partial_X^sMB^{3}_{\en}\right\|^2_{L^2H^{N}}\\
		&=\frac{1}{2}\|\partial_X^sMB^3_{\en}(t_0)\|^2_{H^{N'}}+\delta_0\nu^{1/3}\|\partial_X^sMB^3_{\en}\|_{L^2H^{N'}}^2-\int_{t_0}^{\et}\left\langle \partial_X^sM(U\cdot\nabla_LB^3)_{\en},\partial_X^sMB^3_{\en}\right\rangle_{H^{N}}dt\\
		&\quad+\int_{t_0}^{\et}\left\langle \partial_X^sM(B\cdot\nabla_LU^3)_{\en},\partial_X^sMB^3_{\en}\right\rangle_{H^{N}}dt\\
		&=:\frac{1}{2}\|\partial_X^sMU^3_{\en}(t_0)\|^2_{H^{N}}+\delta_0\nu^{1/3}\|\partial_X^sMU^3_{\en}\|_{L^2H^{N}}^2+L_\lambda+NLT+NLS.
	\end{align*}
	Considering $NLS^{LH}(1,\en,\en)$, for $s=1$, we have 
	\begin{align*}
		NLS^{LH}(1,\en,\en)=&\int_{t_0}^{\et}\left\langle \partial_XM(B^{1,Lo}_{\en}\partial_XU^{3,Hi}_{\en}),\partial_XMB^3_{\en}\right\rangle_{H^{N}}dt\\
		\lesssim&\|MB^1_{\en}\|_{L^2H^N}\|\partial_{XX}MU^3_{\en}\|_{L^2H^N}\|\partial_XMB^3_{\en}\|_{L^\infty H^N}\\
		\lesssim&\ \nu^{-1/2}\epsilon\nu^{-1/2}\epsilon^2=\nu^{-1}\epsilon^3.
	\end{align*}
	While $s=2$, it holds that
	\begin{align*}
		NLS^{LH}(1,\en,\en)=&\int_{t_0}^{\et}\left\langle \partial_{XX}M(B^{1,Lo}_{\en}\partial_XU^{3,Hi}_{\en}),\partial_{XX}MB^3_{\en}\right\rangle_{H^{N}}dt\\
		\lesssim&\|MB^1_{\en}\|_{L^2H^N}\|\partial_{XXX}MU^3_{\en}\|_{L^2H^N}\|\partial_{XX}MB^3_{\en}\|_{L^\infty H^N}\\
		\lesssim&\ \nu^{-1/2}\epsilon\nu^{-5/6}\epsilon\nu^{-1/3}\epsilon=\nu^{-1}\epsilon(\nu^{-1/3}\epsilon)^2.
	\end{align*}
	The estimates of the other contributions of $NLS$ and $NLT$ can be treated similarly to \eqref{bsz3}. 
	
	Turning to $G^3_{\en}$, for $N'\in\{N-2,N\}$, an energy estimate gives
	\begin{align*}
		&\frac{1}{2}\|MG^{3}_{\en}(\et)\|^2_{H^{N'}}+\nu\|\nabla_LMG^{3}_{\en}\|^2_{L^2H^{N'}}+\sum_{i=1}^{2}\left\|\sqrt{-\frac{\partial_t M_i}{M_i}}MG^{3}_{\en}\right\|^2_{L^2H^{N'}}\\
		&=\frac{1}{2}\|MG^{3}_{\en}(t_0)\|^2_{H^{N'}}+\delta_0\nu^{1/3}\|MG^{3}_{\en}\|_{L^2H^{N'}}^2-\int_{t_0}^{\et}\left\langle2\partial^L_{XY}MB^{3}_{\en},MG^{3}_{\en}\right\rangle_{H^{N'}} dt\\
		&\quad-\int_{t_0}^{\et}\left\langle M(U\cdot\nabla_LG^{3})_{\en},MG^{3}_{\en}\right\rangle_{H^{N'}} dt-\int_{t_0}^{\et}\left\langle M(Q\cdot\nabla_LB^{3})_{\en},MG^{3}_{\en}\right\rangle_{H^{N'}} dt\\
		&\quad-2\int_{t_0}^{\et}\left\langle M(\partial^L_iU^{j}\partial^L_{ij}B^{3})_{\en},MG^{3}_{\en}\right\rangle_{H^{N'}} dt+\int_{t_0}^{\et}\left\langle M(B\cdot\nabla_LQ^{3})_{\en},MG^{3}_{\en}\right\rangle_{H^{N'}} dt\\
		&\quad+\int_{t_0}^{\et}\left\langle M(B\cdot\nabla_LQ^{3})_{\en},MG^{3}_{\en}\right\rangle_{H^{N'}} dt+2\int_{t_0}^{\et}\left\langle M(\partial^L_iB^{j}\partial^L_{ij}U^{3})_{\en},MG^{3}_{\en}\right\rangle_{H^{N'}} dt.
	\end{align*}
	The amplification is caused by the nonlinear term $B^{1,Lo}_{\en}\partial_XQ^{3,Hi}_{\en}$, which follows
	\begin{align*}
		\int_{t_0}^{\et}\left\langle M(B^{1,Lo}_{\en}\partial_XQ^{3,Hi}_{\en}),MG^3_{\en}\right\rangle_{H^{N}}dt\lesssim&\|MB^1_{\en}\|_{L^2H^N}\|\partial_XMQ^3_{\en}\|_{L^2H^N}\|MG^3_{\en}\|_{L^\infty H^N}\\
		\lesssim&\ \nu^{-1/2}\epsilon\nu^{-3/2}\epsilon\nu^{-1}\epsilon\lesssim\nu^{-1}\epsilon(\nu^{-1}\epsilon)^2.
	\end{align*}
	The other estimates are similar to those for $Q^3_{\en}$. Therefore, we omit the treatment of $G^3_{\en}$ to conclude the improvement of \eqref{bsb3}--\eqref{bsh3}.
	\section{Energy estimates on zero mode}
	In this section we improve \eqref{bs0}. We provide the details only for $F^1_0,(1,\partial_Z)^2W^1_0$, and the estimates of $F^3_0$ can be treated similarly. For $F^2_0$, the incompressibility implies that $F^2_0$ always has a nonzero $Z$-frequency, which ensures that there is no interaction between two homogeneous modes in the energy estimate of $F^2_0$. Hence, the treatment of $F^2_0$ is similar to $F^2_{\es}$ \eqref{bsf2}.
	\subsection{Estimate of $F_0$}\label{secf0}
	In this subsection, we improve \eqref{bsf0}. An energy estimate gives
	\begin{align*}
		&\frac{1}{2}\|MF^{+,1}_{0}(\et)\|^2_{H^N}+\nu\|\nabla_LMF^{+,1}_{0}\|^2_{L^2H^N}+\left\|\sqrt{-\Upsilon}MF^{+,1}_{0}\right\|^2_{L^2H^N}\\
		&=\frac{1}{2}\|MF^{+,1}_{0}(t_0)\|^2_{H^N}-\int_{t_0}^{\et}\left\langle T^{-t}_{2\alpha}MF^{-,2}_{0},MF^{+,1}_0\right\rangle_{H^N} dt-\int_{t_0}^{\et}\left\langle M((T^{-t}_{2\alpha}W^{-})\cdot\nabla_LF^{+,1})_{0},MF^{+,1}_{0}\right\rangle_{H^N} dt\\
		&\quad-\int_{t_0}^{\et}\left\langle M((T^{-t}_{2\alpha}F^{-})\cdot\nabla_LW^{+,1})_{0},MF^{+,1}_{0}\right\rangle_{H^N} dt-2\int_{t_0}^{\et}\left\langle M(\partial^L_i(T^{-t}_{2\alpha}W^{-,j})\partial^L_{ij}W^{+,1})_{0},MF^{+,1}_{0}\right\rangle_{H^N} dt\\
		&=:\frac{1}{2}\|MF^{+,1}_{0}(t_0)\|_{H^N}^2+LU+NLT+NLS_1+NLS_2.	
	\end{align*}
	The estimate of $LU$ is similar to that in \cite{L20}. For the sake of completeness, we give a sketch of proof. By integration by parts in time, we obtain
	\begin{align*}
		LU=&-\frac{1}{2\alpha}\int_{t_0}^{\et}\left\langle \partial_t(\partial_Z^{-1}T^{-t}_{2\alpha})\mathbb{P}_{l\neq 0}MF^{-,2}_0,MF^{+,1}_0\right\rangle_{H^N}\\
		\leq&\ \frac{1}{2\alpha}\left\|\left\langle \partial_Z^{-1}T^{-t}_{2\alpha}\mathbb{P}_{l\neq 0}MF^{-,2}_0,MF^{+,1}_0\right\rangle_{H^N}\right\|_{L^\infty_t}+\frac{1}{2\alpha}\left|\int_{t_0}^{\et}\left\langle\partial_Z^{-1}T^{-t}_{2\alpha}\mathbb{P}_{l\neq 0}MF^{-,2}_0,\partial_t(MF^{+,1}_0)\right\rangle_{H^N}dt\right|\\
		&+\frac{1}{2\alpha}\left|\int_{t_0}^{\et}\left\langle\partial_Z^{-1}T^{-t}_{2\alpha}\mathbb{P}_{l\neq 0}\partial_t(MF^{-,2}_0),MF^{+,1}_0\right\rangle_{H^N}dt\right|\\
		=&:LU_1+LU_2+LU_3.
	\end{align*}
	For $LU_1$, noting $|\alpha|>8p$, it holds that
	\begin{equation*}
		LU_1\leq\frac{1}{2\alpha}\|MF^1_0\|_{L^\infty H^N}\|MF^2_0\|_{L^\infty H^N}\leq4(\nu^{-2/3}\epsilon)^2.
	\end{equation*}
	Turning next to $LU_2$, expanding the time derivative, we obtain
	\begin{align*}
		LU_2\leq&\ \frac{1}{2\alpha}\left|\int_{t_0}^{\et}\left\langle\partial_Z^{-1}T^{-t}_{2\alpha}MF^{-,2}_0,\frac{\partial_tM}{M}MF^{+,1}_0\right\rangle_{H^N}dt\right|+\frac{1}{2\alpha}\left|\int_{t_0}^{\et}\left\langle\partial_Z^{-1}T^{-t}_{2\alpha}MF^{-,2}_0,T^{-t}_{2\alpha}MF^{-,2}_0\right\rangle_{H^N}dt\right|\\
		&+\frac{1}{2\alpha}\left|\int_{t_0}^{\et}\left\langle\partial_Z^{-1}T^{-t}_{2\alpha}MF^{-,2}_0,\nu\Delta MF^{+,1}_0\right\rangle_{H^N}dt\right|+\frac{1}{2\alpha}\left|\int_{t_0}^{\et}\left\langle\partial_Z^{-1}T^{-t}_{2\alpha}MF^{-,2}_0, M(\partial_tF^{+,1}_0)_{\mathcal{NL}}\right\rangle_{H^N}dt\right|,
	\end{align*}
	where \begin{equation*}
		(\partial_tF^{+,1}_0)_{\mathcal{NL}}=((T^{-t}_{2\alpha}W^{-})\cdot\nabla_LF^{+,1}+(T^{-t}_{2\alpha}F^{-})\cdot\nabla_LW^{+,1}+2\partial^L_i(T^{-t}_{2\alpha}W^{-,j})\partial^L_{ij}W^{+,1})_0.
	\end{equation*}  
	Crucially, the second linear term vanishes. In fact,
	it holds that
	\begin{equation*}
		\left\langle\partial_Z^{-1}T^{-t}_{2\alpha}MF^{-,2}_0,T^{-t}_{2\alpha}MF^{-,2}_0\right\rangle_{H^N}=\int\partial_Z(\partial_Z^{-1}\mathbb{P}_{l\neq 0}T^{-t}_{2\alpha}\langle\nabla\rangle MF^{-,2}_0)^2dXdYdZ=0.
	\end{equation*}
	The other contributions and $LU_3$ can be treated similarly to the $OLS$ term in \eqref{bssymv2}. We omit further details.

	For $NLT$, we start with the interaction between two \nm
	\begin{equation*}
		NLT(i,\en,\en)=\int_{t_0}^{\et}\left\langle M(W^{i}_{\en}\partial_i^LF^{1}_{\en})_{0},MF^{+,1}_{0}\right\rangle_{H^N} dt,
	\end{equation*}
	with $i\in\{1,2,3\}$. When $i\in\{1,3\}$, using divergence-free condition, we have
	\begin{align*}
		NLT(i,\en,\en)\lesssim&\|M(W^1_{\en},W^3_{\en})\|_{L^2H^N}(\|\partial^L_YMF^2_{\en}\|_{L^2H^N}+\|MF^3_{\en}\|_{L^2H^N})\|MF^1_0\|_{L^\infty H^N}\\
		\lesssim&\ \nu^{-1/2}\epsilon\nu^{-7/6}\epsilon\nu^{-2/3}\epsilon=\nu^{-1}\epsilon(\nu^{-2/3}\epsilon)^2.
	\end{align*}
	Turning to $i=2$, by integration by parts, it holds that
	\begin{align*}
		NLT^{LH}(2,\en,\en)=&\int_{t_0}^{\et}\left\langle \partial^L_YM((T^{-t}_{2\alpha}W^{Lo,2}_{\en})F^{Hi,1}_{\en})-M((T^{-t}_{2\alpha}\partial^L_YW^{Lo,2}_{\en})F^{Hi,1}_{\en}),MF^{+,1}_{0}\right\rangle_{H^N} dt\\
		\lesssim&(\|MW^2_{\en}\|_{L^\infty H^{3/2+}}\|\partial_YMF^1_0\|_{L^2H^N}+\|\partial_Y^LMW^2_{\en}\|_{L^\infty H^{N}}\|MF^1_0\|_{L^\infty H^N})\|MF^1_{\en}\|_{L^2H^N}\\
		\lesssim&\ \epsilon\nu^{-7/6}\epsilon\nu^{-7/6}\epsilon+\nu^{-1/2}\epsilon\nu^{-2/3}\epsilon\nu^{-7/6}\epsilon=\nu^{-1}\epsilon(\nu^{-2/3}\epsilon)^2.
	\end{align*}
	For $NLT^{HL}(2,\en,\en)$, expanding $W^+=T^{-t}_{\alpha}(U+B)$ and $F^1_{\en}=-\partial^L_Y\partial_X^{-1}F^2_{\en}+\sigma F^3_{\en}$, we get
	\begin{align*}
		NLT^{HL}(2,\en,\en)=&\int_{t_0}^{\et}\big\langle M(U^{2,Hi}_{\en}\partial^L_Y\partial^L_Y\partial_X^{-1}F^{2,Lo}_{\en}),MF^1_{0}\big\rangle_{H^N}dt\\
		&+\int_{t_0}^{\et}\big\langle M(B^{2,Hi}_{\en}\partial^L_Y\partial^L_Y\partial_X^{-1}F^{2,Lo}_{\en}),MF^1_{0}\big\rangle_{H^N}dt\\
		&+\int_{t_0}^t\big\langle M(W^{2,Hi}_{\en}\partial^L_YF^{3,Lo}_{\en}),MF^1_{0}\big\rangle_{H^N}dt=:I_1+I_2+I_3.
	\end{align*}
	Noting $|\xi-kt|^2\lesssim\nu^{-2/3}e^{\delta_0\nu^{1/3}t}|k,\xi|^2$ and $N>9/2$, by H$\ddot{\mathrm{o}}$lder's inequality and Fubini's theorem, we obtain
	\begin{align*}
		&\big\langle M(U^{2,Hi}_{\en}\partial^L_Y\partial^L_Y\partial_X^{-1}F^{2,Lo}_{\en}),MF^1_{0}\big\rangle_{H^N}\\
		&\lesssim\nu^{-2/3}\sum_{l,k\neq0}
		\iint_{\mathbb{R}^2}\left|\langle k,\eta-\xi,l\rangle^NM\hat{U}^2_{\en}(-k,\eta-\xi,-l)|k,\xi|^2M\hat{F}^2_{\en}(k,\xi,l)\langle\eta\rangle^NM\bar{\hat{F}}^1(0,\eta,0)\right|d\eta d\xi\\
		&\lesssim\nu^{-2/3}\left(\sum_{l,k\neq0}\iint_{\mathbb{R}^2}\left|\langle k,\eta-\xi,l\rangle^N|k,\eta-\xi+kt|M\hat{U}^2_{\en}(-k,\eta-\xi,-l)\langle|k,\xi|\rangle^3 M\hat{F}^2_{\en}(k,\xi,l)\right|^2d\eta d\xi\right)^{1/2}\\
		&\quad\times\left(\sum_{k\neq0}\iint_{\mathbb{R}^2}\left|\frac{\langle\eta\rangle^N}{|k,\eta-\xi+kt|\langle|k,\xi|\rangle}M\hat{F}^1(0,\eta,0)\right|^2d\xi d\eta\right)^{1/2}\\
		&\lesssim\ \nu^{-2/3}\||\nabla_L|M U^2_{\en}\|^2_{H^N}\|MF^2_{\en}\|_{H^N}\left(\int_{\mathbb{R}}\sum_{k\neq0}\frac{1+2|k|}{|k|[(1+2|k|^2)+(\eta+kt)^2]}\langle\eta\rangle^{2N}|M\hat{F}^1(0,\eta,0)|^2d\eta\right)^{1/2}\\
		&\lesssim\nu^{-2/3}\||\nabla_L|M U^2_{\en}\|^2_{H^N}\|MF^2_{\en}\|_{H^N}\|\sqrt{-\Upsilon}MF^1_0\|_{H^N},
	\end{align*}
	where we used the fact, for $a,b>0,c\in\mathbb{R}$, it holds that
	\begin{equation*}
		\int_{\mathbb{R}}\frac{\mathrm{d}\xi}{(a^2+\xi^2)(b^2+(c-\xi)^2)}=\frac{\pi}{ab}\frac{a+b}{(a+b)^2+c^2}.
	\end{equation*}
	Integrating in time, we have
	\begin{equation*}
		I_1\lesssim\nu^{-2/3}\||\nabla_L|M U^2_{\en}\|^2_{L^2H^N}\|MF^2_{\en}\|_{L^\infty H^N}\|\sqrt{-\Upsilon}MF^1_0\|_{L^2H^N}\lesssim\nu^{-1}\epsilon(\nu^{-2/3}\epsilon)^2.
	\end{equation*}
	Turning next to $I_2$ and $I_3$, we have
	\begin{align*}
		I_2\leq&\int_{t_0}^{\et}\left|\big\langle M(B^{2,Hi}_{\en}\partial^L_Y\partial^L_Y\partial_X^{-1}F^{2,Lo}_{\en}),MF^1_{0}\big\rangle_{H^N}\right|dt\\
		\lesssim&\ \nu^{-2/3}\|MB^2_{\en}\|_{L^2H^N}\|MF^2_{\en}\|_{L^2H^N}\|MF^1_{0}\|_{L^\infty H^N}\\
		\lesssim&\ \nu^{-2/3}\nu^{-1/6}\epsilon\nu^{-5/6}\epsilon\nu^{-2/3}\epsilon=\nu^{-1}\epsilon(\nu^{-2/3}\epsilon)^2,
	\end{align*}
	as well as 
	\begin{align*}
		I_3\leq&\int_{t_0}^t\left|\big\langle M(W^{2,Hi}_{\en}\partial^L_YF^{3,Lo}_{\en}),MF^1_{0}\big\rangle_{H^N}\right|dt\\
		\lesssim&\|MW^2_{\en}\|_{L^2H^N}\|\partial_Y^LF^3_{\en}\|_{L^2H^{3/2+}}\|MF^1_0\|_{L^\infty H^N}\\
		\lesssim&\ \nu^{-1/3}\epsilon\nu^{-7/6}\epsilon\nu^{-2/3}\epsilon=\nu^{-5/6}\epsilon(\nu^{-2/3}\epsilon)^2.
	\end{align*}
	Putting together the above estimates, we finish the improvement of $NLT(\en,\en)$. The interactions between \sm and \nm, or two non-homogeneous modes can be treated similarly and hence are omitted for brevity. For $NLT(0,0)$, the dissipation in \eqref{bs0} implies
	\begin{align*}
		NLT(0,0)=&\int_{t_0}^{\et}\left\langle M(W_{0}\cdot\nabla F^{1}_{0}),MF^{+,1}_{0}\right\rangle_{H^N} dt\\
		\lesssim&\|\nabla MW^2_0\|_{L^2H^N}\|\nabla MF^1_0\|_{L^2H^N}\|MF^1_0\|_{L^\infty H^N}\\
		&+\|\nabla MW^3_0\|_{L^2H^N}\|\partial_ZMF^1_0\|_{L^2H^N}\|MF^1_0\|_{L^\infty H^N}+\| MW^3_0\|_{L^\infty H^N}\|\partial_ZMF^1_0\|_{L^2H^N}\|\nabla MF^1_0\|_{L^2 H^N}\\
		\lesssim&\ \nu^{-1}\epsilon(\nu^{-2/3}\epsilon)^2.
	\end{align*}
	Turning to $NLS_1(\en,\en)$, by divergence-free condition, it holds that
	\begin{align*}
		NLS_1(\en,\en)=&\int_{t_0}^{\et}\left\langle
		 M(F_{\en}\cdot\nabla_LW^{1}_{\en})_{0},MF^{+,1}_{0}\right\rangle_{H^N} dt\\
		\lesssim&\|(MF^1_{\en},MF^3_{\en})\|_{L^2H^N}\|\partial_XMW^1_{\en}\|_{L^2H^N}\|MF^1_0\|_{L^\infty H^N}\\
		&+\|MF^2_{\en}\|_{L^2H^N}(\|MF^2_{\en}\|_{L^2H^N}+\|\partial_Y^L\partial_XMW^3_{\en}\|_{L^2H^N})\|MF^1_0\|_{L^\infty H^N}\\
		\lesssim&\ \nu^{-7/6}\epsilon\nu^{-1/2}\epsilon\nu^{-2/3}\epsilon+\nu^{-5/6}\epsilon\nu^{-5/6}\epsilon\nu^{-2/3}\epsilon=\nu^{-1}\epsilon(\nu^{-2/3}\epsilon)^2.
	\end{align*}
	The estimates of the other contributions of $NLS_1$ and $NLS_2$ are similar to the above, and we omit further details. 
	
	Finally, the estimate of $F^3_0$ is natural from the above. For example, turning to the interaction between $B_{\en}$ and $\nabla_LF^{3}_{\en}$, it holds that
	\begin{align*}
		\int_{t_0}^{\et}\left\langle M(B_{\en}\cdot\nabla_LF^{3}_{\en})_0,MF^3_{0}\right\rangle_{H^N}dt\lesssim&\|\partial_X(MB^1_{\en},MB^3_{\en})\|_{L^2H^N}\|MF^3_{\en}\|_{L^2H^N}\|MF^3_0\|_{L^\infty H^N}\\
		&+\|MB^2_{\en}\|_{L^2H^N}\|\partial_Y^LMF^3_{\en}\|_{L^2H^N}\|MF^3_0\|_{L^\infty H^N}\\
		\lesssim&\ \nu^{-1/2}\epsilon\nu^{-7/6}\epsilon\nu^{-2/3}\epsilon+\nu^{-1/6}\epsilon\nu^{-3/2}\epsilon\nu^{-2/3}\epsilon\lesssim\nu^{-1}\epsilon(\nu^{-2/3}\epsilon)^2.
	\end{align*} 
  	Noting that there is no interaction between two homogeneous modes in nonlinear pressure term, it can be treated similarly to \eqref{bsf3}. This completes the improvement of \eqref{bsf0}.
	\subsection{Estimate of $(1,\partial_Z)^2W_0$}
	In this subsection, we improve $\eqref{bsz0}$. We start with $\partial_{ZZ}W^1_0$, an energy estimate gives 
	\begin{align*}
		&\frac{1}{2}\|\partial_{ZZ}MW^{+,1}_0(\et)\|_{H^N}+\nu\|\nabla_L\partial_{ZZ}MW^{+,1}_{0}\|^2_{L^2H^N}+\left\|\sqrt{-\Upsilon}\partial_{ZZ}MW^{+,1}_{0}\right\|^2_{L^2H^N}\\
		&=\frac{1}{2}\|\partial_{ZZ}MW^{+,1}_{0}(t_0)\|^2_{H^N}-\int_{t_0}^{\et}\left\langle T^{-t}_{2\alpha}\partial_{ZZ}MW^{-,2}_{0}+\partial_{ZZ}M((T^{-t}_{2\alpha}W^{-})\cdot\nabla_LW^{+,1})_{0},\partial_{ZZ}MW^{+,1}_0\right\rangle_{H^N} dt\\
		&=:\frac{1}{2}\|\partial_{ZZ}MW^{+,1}_{0}(t_0)\|_{H^N}^2+LU+NLT.	
	\end{align*}
	The $LU$ term can be dealt with using integration by parts in time as in section \ref{secf0}. We skip it and turn to $NLT$.
	Since $\partial_{ZZ}W_0$ has a nonzero $Z$-frequency, there is no interaction between two homogeneous modes. For $NLT(\en,\es)$, we define the components
	\begin{equation*}
		NLT(i,\en,\es)=\int_{t_0}^{\et}\left\langle \partial_{ZZ}M(W^i_{\en}\cdot\partial_i^LW^{1}_{\es})_{0},\partial_{ZZ}MW^{+,1}_{0}\right\rangle_{H^N} dt,
	\end{equation*}
	with $i\in\{1,2,3\}$.
	By the bootstrap assumptions, it holds that
	\begin{align*}
		NLT(1,\en,\es)=&NLT^{HL}(1,\en,\es)+NLT^{LH}(1,\en,\es)\\
		\lesssim&\|\partial_{XX}MW^1_{\en}\|_{L^2H^N}\|\partial_XMW^1_{\es}\|_{L^2H^N}\|\partial_{ZZ}MW^1_0\|_{L^\infty H^N}\\
		&+\|MW^1_{\en}\|_{L^2H^{3/2+}}\|\partial_{ZZX}MW^1_{\es}\|_{L^2H^N}\|\partial_{ZZ}MW^1_0\|_{L^\infty H^N}\\
		\lesssim&\ \nu^{-1/2}\epsilon\nu^{-1/6}\epsilon^2+\nu^{-1/2}\epsilon\nu^{-1/2}\epsilon\epsilon\lesssim\nu^{-1}\epsilon^3.
	\end{align*}
	When $j\in\{2,3\}$, using integration by parts, we obtain
	\begin{align*}
		NLT(j,\en,\es)=&NLT^{HL}(j,\en,\es)+NLT^{LH}(j,\en,\es)\\
		\lesssim&\|\partial_{XX}MW^j_{\en}\|_{L^2H^N}\|\nabla_LMW^1_{\es}\|_{L^2H^N}\|\partial_{ZZ}MW^1_0\|_{L^\infty H^N}\\
		&+\|MW^j_{\en}\|_{L^\infty H^{3/2+}}\|\nabla_L\partial_ZMW^1_{\es}\|_{L^2H^N}\|\partial_Z\partial_{ZZ}MW^1_0\|_{L^2 H^N}\\
		\lesssim&\ \nu^{-1/2}\epsilon\nu^{-1/2}\epsilon\epsilon+\epsilon\nu^{-1/2}\epsilon\nu^{-1/2}\epsilon=\nu^{-1}\epsilon^3.
	\end{align*}
	Turning to $NLT(\es,\en)$, we have
	\begin{align*}
		NLT(\es,\en)=&NLT^{HL}(\es,\en)+NLT^{LH}(\es,\en)\\
		\lesssim&\|\partial_{ZZ}(MW^1_{\es},MW^3_{\es})\|_{L^2 H^N}\|\partial_XMW^1_{\en}\|_{L^2H^{3/2+}}\|\partial_{ZZ}MW^1_0\|_{L^\infty H^N}\\
		&+\|\partial_{ZZ}MW^2_{\es}\|_{L^2 H^N}\|\partial_Y^LMW^1_{\en}\|_{L^2H^{3/2+}}\|\partial_{ZZ}MW^1_0\|_{L^\infty H^N}\\
		&+\|\partial_X(MW^1_{\es},MW^3_{\es})\|_{L^\infty H^{3/2+}}\|\partial_{XX}MW^1_{\en}\|_{L^2H^N}\|\partial_{ZZ}MW^1_0\|_{L^2H^N}\\
		&+\|\partial_XMW^2_{\es}\|_{L^2 H^{3/2+}}\|\partial_{XY}^LMW^1_{\en}\|_{L^2H^N}\|\partial_{ZZ}MW^1_0\|_{L^\infty H^N}\\
		\lesssim&\ \nu^{-1/2}\epsilon\nu^{-1/2}\epsilon^2+\nu^{-1/6}\epsilon\nu^{-5/6}\epsilon^2+\nu^{-1/6}\epsilon\nu^{-2/3}\epsilon\epsilon+\epsilon\nu^{-1}\epsilon^2\lesssim\nu^{-5/6}\epsilon^3.
	\end{align*}
	For $NLT(0,0)$, we use the integration by parts to obtain
	\begin{align*}
		NLT(0,0)\lesssim&\|\partial_Z(MW^2_{0},MW^3_0)\|_{L^\infty H^N}\|\nabla_L\partial_ZMW^1_{0}\|_{L^2H^N}\|\partial_Z\partial_{ZZ}MW^1_0\|_{L^2H^N}\lesssim\nu^{-1}\epsilon^3.
	\end{align*}
	Finally, for $W_0^{+,r}$ with $r\in\{1,2,3\}$, the energy estimate reads
	\begin{align*}
		&\frac{1}{2}\|MW^{+,r}_0(\et)\|_{H^N}+\nu\|\nabla_LMW^{+,r}_{0}\|^2_{L^2H^N}+\left\|\sqrt{-\Upsilon}MW^{+,r}_{0}\right\|^2_{L^2H^N}\\
		&=\frac{1}{2}\|MW^{+,r}_{0}(t_0)\|^2_{H^N}-\int_{t_0}^{\et}\left\langle M((T^{-t}_{2\alpha}W^{-})\cdot\nabla_LW^{+,r})_{0},MW^{+,r}_{0}\right\rangle_{H^N} dt\\
		&\quad-\mathbf{1}_{r=1}\int_{t_0}^{\et}\left\langle T^{-t}_{2\alpha}MW^{-,2}_{0},MW^{+,r}_0\right\rangle_{H^N} dt+\mathbf{1}_{r\neq1}\int_{t_0}^{\et}\left\langle \partial_r\Delta^{-1}M((T^{-t}_{2\alpha}\partial_jW^{-,i})\partial_iW^{+,j})_{0},MW^{+,r}_{0}\right\rangle_{H^N} dt	.
	\end{align*}
	The estimate is similar to that of $\eqref{bsf0}$, and one can proceed analogously with $W^1_0$ playing the role of $F^1_0$, which completes the improvement of \eqref{bsz0}.
		
	\vspace{4 mm}
	\noindent \textbf{Acknowledgements:} The authors would like to thank Prof. Zhifei Zhang for pointing out this question to us and the meaningful discussions. The work of FW is supported by the National Natural Science Foundation of China (No. 12471223 and 12331008). Lingda Xu is supported by the Research Centre for Nonlinear Analysis at The Hong Kong Polytechnic University.
	\bibliography{ref}
\end{document}